\documentclass{amsart}
\usepackage{amssymb, amsmath, latexsym, mathabx, dsfont, tikz, multirow}

\renewcommand{\baselinestretch}{\baselinestretch}
\renewcommand{\baselinestretch}{1.1}
\numberwithin{equation}{section}

\newtheorem{thm}{Theorem}[section]
\newtheorem{lem}[thm]{Lemma}

\theoremstyle{definition}

\theoremstyle{remark}
\newtheorem{rmk}[thm]{Remark}

\numberwithin{equation}{section}

\newcommand{\Mod}[1]{\ (\mathrm{mod}\ #1)}

\begin{document}

\title{Universal mixed sums of generalized $4$- and $8$-gonal numbers}
\author{Jangwon Ju and Byeong-Kweon Oh }

\address{Research institute of Mathematics, Seoul National University, Seoul 08826, Korea}
\email{jjw@snu.ac.kr}
\thanks{This work of the first author was supported by Basic Science Research Program through the National Research Foundation of Korea(NRF) funded by the Ministry of Education(NRF-2018R1A6A3A01012662)}

\address{Department of Mathematical Sciences and Research Institute of Mathematics, Seoul National University, Seoul 08826, Korea}
\email{bkoh@snu.ac.kr}
\thanks{This work of the second author was supported by the National Research Foundation of Korea (NRF-2017R1A2B4003758).}



\begin{abstract}  
An integer of the form $P_m(x)= \frac{(m-2)x^2-(m-4)x}{2}$ for an integer $x$, is called a generalized $m$-gonal number. 
For positive integers $\alpha_1,\dots,\alpha_u$ and $\beta_1,\dots,\beta_v$, a mixed sum $\Phi=\alpha_1P_4(x_1)+\cdots+\alpha_uP_4(x_u)+\beta_1P_8(y_1)+\cdots+\beta_vP_8(y_v)$ of generalized $4$- and $8$-gonal numbers is called universal if $\Phi=N$ has an integer solution for any nonnegative integer $N$.
 In this article, we prove that there are exactly 1271 proper universal mixed sums of generalized $4$- and $8$-gonal numbers. 
Furthermore, the ``$61$-theorem" is proved, which states that an arbitrary mixed sum of generalized $4$- and $8$-gonal numbers is universal if and only if it represents the integers $1$, $2$, $3$, $4$, $5$, $6$, $7$, $8$, $9$, $10$, $12$, $13$, $14$, $15$, $18$, $20$, $30$, $60$,  and $61$.
\end{abstract}

\maketitle

\section{Introduction}
The famous Lagrange's four-square theorem says that every nonnegative integer can be written as a sum of at most four squares of integers. 
Motivated by this celebrated four-square theorem, 
Ramanujan provided a list of $55$ candidates of diagonal quaternary integral quadratic forms that represent all nonnegative integers. 
 Dickson pointed out that the diagonal quaternary quadratic form $x^2+2y^2+5z^2+5t^2$ in the Ramanujan's list does not represent the integer $15$, and confirmed that Ramanujan's assertion is correct for all the other $54$ quaternary forms (for details, see \cite{dick}).  
 Recently, Conway, Miller, and Schneeberger proved the so called ``15-theorem", which states that a positive definite integral quadratic form represents  all nonnegative integers if and only if it represents the integers $1,2,3,5,6,7,10,14,~\text{and}~15$, irrespective of its rank (for details, see \cite{b}).
As a natural generalization of the 15-theorem,  Bhargava and Hanke \cite{BH} proved the so-called ``290-theorem", which states that    
every positive-definite integer-valued quadratic form represents all nonnegative integers if and only if it represents 
$$
\begin{array} {ll}
1,\ 2, \ 3,\ 5,\ 6,\ 7,\ 10,\ 13,\ 14,\ 15,\ 17,\ 19,\ 21,\ 22,\ 23,\ 26,\ 29,\\
30,\ 31,\ 34,\ 35,\ 37,\ 42,\ 58,\ 93,\ 110,\ 145,\ 203, \ \  \text{and} \  \ 290.
\end{array}
$$
Here a quadratic form $f(x_1,x_2,\dots,x_n)=\sum_{1 \le i, j\le n} a_{ij} x_ix_j \ (a_{ij}=a_{ji})$ 
is called {\it integral} if $a_{ij}\in \mathbb Z$ for any $i,j$, and is called {\it integer-valued} if $a_{ii}\in \mathbb Z$ and $a_{ij}+a_{ji}\in\mathbb Z$ for any $i,j$.

For an integer $m\geq3$, a {\it polygonal number of order $m$} (or {\it an $m$-gonal number}) is defined by 
$$
P_m(x)= \frac{(m-2)x^2-(m-4)x}{2},
$$
for some nonnegative integer $x$.  If we admit that $x$ is a  negative integer, then $P_m(x)$ is called  
{\it a generalized polygonal number of order $m$} (or {\it a generalized $m$-gonal number}). 
Lagrange's four-square theorem implies that 
the equation 
$$
P_4(x)+P_4(y)+P_4(z)+P_4(t)=N
$$
has an integer solution $(x,y,z,t)\in\mathbb{Z}^4$, for any nonnegative integer $N$. 

As a natural generalization of Ramanujan's result on diagonal quaternary quadratic forms representing all nonnegative integers, it was proved in \cite{octagonal} and \cite{sun} that there are exactly $40$ quaternary sums of generalized octagonal numbers which represent all nonnegative integers. 
Furthermore,  the ``octagonal theorem of sixty" was proved in \cite{octagonal}, which states that  a sum $a_1 P_8(x_1)+a_2P_8(x_2)+\cdots+a_kP_8(x_k)$ of  generalized octagonal numbers represents all nonnegative integers if and only if it represents 
$$
1,\ 2, \ 3,\  4,\ 6,\  7,\ 9,\ 12,\ 13,\ 14,\ 18, \ \text{and} \ \  60,
$$ 
which might be considered this as a natural generalization of  the ``$15$-theorem".

In this article, we generalize the above theorem to the mixed sum of generalized $4$- and $8$-gonal numbers.
Let $\alpha_1,\dots,\alpha_u$, and $\beta_1,\dots,\beta_v$ be positive integers. The mixed sum
$$
\Phi=\alpha_1P_{4}(x_1)+\cdots+\alpha_uP_{4}(x_u)+\beta_1P_8(y_1)+\cdots+\beta_vP_{8}(y_v) 
$$
of generalized $4$- and $8$-gonal numbers is said to be {\it universal} if the diophantine equation $\Phi=N$ has an integer solution for any nonnegative integer $N$.  The aim of this article is to prove that there are exactly $1271$ {\it proper universal} mixed sums of generalized $4$- and $8$-gonal numbers. Here,  a universal sum $\Phi$ of generalized $4$- and $8$-gonal numbers is called {\it proper} if there does not exist a proper sum of $\Phi$ that is universal.  Furthermore, in Section 5, we prove that a mixed sum of generalized $4$- and $8$-gonal numbers is universal if and only if it represents the integers
$$
1,\ 2,\ 3, \ 4, \ 5, \ 6, \ 7, \ 8, \ 9, \ 10, \ 12, \ 13, \ 14, \ 15, \ 18, \ 20, \ 30, \ 60, \ \text{and} \ \ 61,
$$
which might also be considered as a generalization of the ``$15$-theorem".

Let 
$$
f(x_1,x_2,\dots,x_n)=\sum_{1 \le i, j\le n} a_{ij} x_ix_j \ (a_{ij}=a_{ji} \in \mathbb Z)
$$
 be a positive definite integral quadratic form. 
The corresponding symmetric matrix of $f$ is defined by $M_f=(a_{ij})$. 
For a diagonal quadratic form $f(x_1,x_2,\dots,x_n)=a_1x_1^2+a_2x_2^2+\cdots+a_nx_n^2$, we simply write 
$$
f=\langle a_1,a_2,\dots,a_n\rangle.
$$ 
For an integer $k$, we say $k$ is {\it represented by $f$} if the equation $f(x_1,x_2,\dots,x_n)=k$ has an integer solution. 
The genus of $f$, denoted by $\text{gen}(f)$, is the set of all quadratic forms that are isometric to $f$ over the $p$-adic integer ring $\mathbb{Z}_p$ for any prime $p$. 
The number of isometry classes in $\text{gen}(f)$ is called the class number of $f$, and denoted by $h(f)$.

Any unexplained notations and terminologies can be found in \cite{ki} or \cite{om}.

\section{General tools}
Let $\alpha_1,\dots,\alpha_u$ and $\beta_1,\dots,\beta_v$ be positive integers such that
$\alpha_1\leq\cdots\leq\alpha_u$ and $\beta_1\leq\cdots\leq\beta_v$. Let
\begin{equation} \label{defphi}
\Phi=\alpha_1P_{4}(x_1)+\cdots+\alpha_uP_{4}(x_u)+\beta_1P_8(y_1)+\cdots+\beta_vP_{8}(y_v) 
\end{equation}
be a mixed sum of generalized $4$- and $8$-gonal numbers. For simplicity, we will use the notation
$$
\Phi=[\![\alpha_1,\dots,\alpha_u,\boldsymbol{\beta_{1}},\dots,\boldsymbol{\beta_{v}}]\!].
$$
Recall that $\Phi$ is universal if the diophantine equation $\Phi=N$ has an integer solution $(x_1,\dots,x_u,y_1,\dots,y_v) \in \mathbb Z^{u+v}$ for any nonnegative integer $N$.  One may directly check that $\Phi$ is universal if and only if  the equation
$$
3\alpha_1x_1^2+\cdots+3\alpha_ux_u^2+\beta_1 (3y_{1}-1)^2+\cdots+\beta_v(3y_v-1)^2=
3N+\beta_1+\cdots+\beta_v
$$
has an integer solution for any nonnegative integer $N$. This is equivalent to the existence of an integer solution $(x_1,\dots,x_u,y_1,\dots,y_v)\in\mathbb{Z}^{u+v}$ of 
\begin{equation} \label{core}
\begin{array} {ll}
3\alpha_1x_1^2+\cdots+3\alpha_ux_u^2\!\!\!\!&+\beta_1y_1^2+\cdots+
\beta_vy_v^2\\
&=3N+\beta_1+\cdots+\beta_v \ \ \text{with} \  \  y_1\cdots y_v\not\equiv0\Mod3.
\end{array}
\end{equation}
Note that Equation \eqref{core} corresponds to the representation of a quadratic form with congruence condition. Since there are some method on determineing the existence of the representations of integers by a quadratic form, we try to find a suitable method on reducing Equation \eqref{core} to the representation of a quadratic form without congruence condition.
To explain our method, for example, assume that $\beta_1+\cdots+\beta_v \not \equiv 0 \Mod 3$. Since at least one of $\beta_i$ is not divisible by $3$, we may assume that, without loss of generality, $\beta_v \not \equiv  0 \Mod 3$. Then \eqref{core} has an integer solution if and only if 
$$
3\alpha_1x_1^2+\cdots+3\alpha_ux_u^2+\beta_1(y_v-3y_1)^2+\cdots+
\beta_{v-1} (y_v-3y_{v-1})^2+\beta_v y_v^2=3N+\beta_1+\cdots+\beta_v
$$
has an integer solution. Hence, in this case, the problem can be reduced to the representations of a quadratic form.  
Now, assume that   $\beta_1+\cdots+\beta_v \equiv 0 \Mod 3$.  
As a representative case, suppose that $u\geq3$, and the set of all integers that are represented by $3\alpha_1x_1^2+\cdots+3\alpha_ux_u^2$ is known.
In this case, we try to find sets $S_i$ of integers not divisible by $3$  such that for each $i=1,2,\dots,v$, there is a $y_i \in S_i$ depending on $N$ such that 
\begin{equation} \label{core2}
3\alpha_1x_1^2+\cdots+3\alpha_ux_u^2=3N+\beta_1+\cdots+\beta_v-(\beta_1y_1^2+\cdots
+\beta_vy_v^2) 
\end{equation}
has an integer solution $(x_1,\dots,x_u)\in\mathbb{Z}^{u}$. 
 
In \cite{poly}, \cite{regular}, and \cite{pentagonal}, we developed a method that determines whether or not integers in an arithmetic progression are represented by some particular ternary quadratic form. 
We briefly introduce this method for those who are unfamiliar with it. 
 
 Let $d$ be a positive integer and let $a$ be a nonnegative integer $(a\leq d)$. 
We define 
$$
S_{d,a}=\{dn+a \mid n \in \mathbb N \cup \{0\}\}.
$$
For integral ternary quadratic forms $f,g$, we define
$$
R(g,d,a)=\{v \in (\mathbb{Z}/d\mathbb{Z})^3 \mid vM_gv^t\equiv a \ (\text{mod }d) \}
$$
and
$$
R(f,g,d)=\{T\in M_3(\mathbb{Z}) \mid  T^tM_fT=d^2M_g \}.
$$
A coset (or, a vector in the coset) $v \in R(g,d,a)$ is said to be {\it good} with respect to $f,g,d, \text{ and }a$ if there is a $T\in R(f,g,d)$ such that $\frac1d \cdot vT^t \in \mathbb{Z}^3$.  
The set of all good vectors in $R(g,d,a)$ is denoted by $R_f(g,d,a)$.   
If  $R(g,d,a)=R_f(g,d,a)$, we write  
$$
g\prec_{d,a} f.
$$ 
If $g\prec_{d,a} f$, then by Lemma 2.2 of \cite{regular},  we have 
\begin{equation}\label{good}
S_{d,a}\cap Q(g) \subset Q(f).
\end{equation}

We denote the set $R(g,d,a)\setminus R_f(g,d,a)$ by $B_f(g,d,a)$.
In general, if $d$ is large, then the set  $B_f(g,d,a)$ has too many vectors to compute it exactly by hand. 
 A MAPLE based  computer program for computing this set  is available upon request to the authors. 

\begin{thm}\label{pme}
Under the same notation given above, assume that  $T\in M_3(\mathbb Z)$ satisfies the following conditions:
\begin{enumerate}
\item[(i)] $\frac1dT$ has an infinite order;
\item[(ii)] $T^tM_gT=d^2M_g$;
\item[(iii)]  for any vector $v \in \mathbb Z^3$ such that  $v\, (\text{mod }d)\in B_f(g,d,a)$, $\frac1d\cdot vT^t\in \mathbb Z^3$. 
\end{enumerate}
Then we have
$$
S_{d,a} \cap Q(g)\setminus \{g(z)\cdot s^2\mid s\in\mathbb{Z}\}\subset Q(f),
$$
where the vector $z$ is any integral primitive eigenvector of $T$. 
\end{thm}

\begin{proof}
See Theorem 2.3 of \cite{poly}.
\end{proof}

\section{Ternary mixed sums of generalized $4$- and $8$-gonal numbers} 
In this section, we determine all ternary universal mixed sums of generalized $4$- and $8$-gonal numbers.

Let $\Phi$ be a mixed sum of generalized $4$- and $8$-gonal numbers defined in \eqref{defphi}.  
For an integer $N$, if the diophantine equation $\Phi=N$ has an integer solution, then we say the  mixed sum $\Phi$ represents $N$, and we write  $N\rightarrow\Phi$. 
When the mixed sum $\Phi$ is not universal,   the least positive integer that is not represented by $\Phi$ is called the {\it truant} of $\Phi$, and is denoted by $t(\Phi)$. One may easily show that there does not exist a binary universal sum of generalized polygonal numbers. For the ternary case, we have the following:

\begin{thm}\label{ternary}
There are exactly $6$ ternary universal mixed sums of generalized $4$- and $8$-gonal numbers. In fact, they are
\begin{equation}\label{ternary uni}
[\![1,1,\boldsymbol1]\!], \ \      [\![1,1,\boldsymbol2]\!], \ \      [\![1,3,\boldsymbol1]\!], \ \      [\![1,\boldsymbol1,\boldsymbol1]\!],  \ \     [\![2,3,\boldsymbol1]\!], \ \  \text{and} \ \  [\![2,\boldsymbol1,\boldsymbol1]\!].     
\end{equation}
\end{thm}

\begin{proof}
Let $u,v$ be nonnegative integers such that $u+v=3$. 
For positive integers $\alpha_1,\dots,\alpha_u$ and $\beta_1,\dots,\beta_v$, assume that a ternary mixed sum
$$
\Phi=[\![\alpha_1,\dots,\alpha_u,\boldsymbol{\beta_1},\dots,\boldsymbol{\beta_v}]\!]
$$ 
of generalized $4$- and $8$-gonal numbers is universal. 
At first, we find all candidates of ternary mixed sum $\Phi$ of generalized $4$-and $8$-goanl numbers by using, so called, the escalation method.  Since $1\rightarrow\Phi$, $\alpha_1=1$ or $\beta_1=1$. If $\alpha_1=1$, then $1\leq\alpha_2\leq2$ or $1\leq\beta_1\leq2$, for $2\rightarrow\Phi$ and $t([\![1]\!])=2$.
If $\alpha_1 \ne 1$ and $\beta_1=1$, then $\alpha_1=2$ or $1\leq\beta_2\leq2$, for $2\rightarrow\Phi$ and $t([\![\boldsymbol1]\!])=2$.
Let $u_1,v_1$ be nonnegative integers such that $u_1+v_1=2$ and
$$
\Phi'=[\![\alpha_1,\dots,\alpha_{u_1},\boldsymbol{\beta_1},\dots,\boldsymbol{\beta_{v_1}}]\!]
$$ 
is one of the above proper sums of $\Phi$. 
Since $t(\Phi')$ is represented by $\Phi$ for each possible case,  at least one of  $\alpha_{u_1+1}$ or $\beta_{v_1+1}$ is less than equal to $t(\Phi')$ given in Table 3.1. 
From this, one may easily show that there are  $42$ candidates of ternary universal mixed sums of generalized $4$- and $8$-gonal numbers (see Table 3.1). Among them, $36$ ternary sums are, in fact, not universal (for truants of these sums, see Table (4.1)). Since the proofs of the universalities of the remaining $6$ candidates are quite similar to each other, we only provide the proof of the case (3-1).

\begin{table}[h]
\centering
\begin{footnotesize}
\renewcommand{\arraystretch}{1}\renewcommand{\tabcolsep}{2mm}
\begin{tabular}{|c|c|c|l|ll|ll|}\multicolumn{8}{c}{\textbf{Table 3.1.} Ternary universal mixed sums}\\
\hline
Case&$\Phi'$&$t(\Phi')$&\multicolumn{1}{c|}{C. A.}&\multicolumn{2}{c|}{Possible Candidates}&\multicolumn{2}{c|}{Universal Cases}\\
\hline
(3-1)&$[\![1,1]\!]$&$3$&~&1$\leq\alpha_3\leq3$,&$1\leq\beta_1\leq3$&~&$\beta_1=1,2$\\ 
(3-2)&$[\![1,2]\!]$&$5$&~&$2\leq\alpha_3\leq5$,&$1\leq\beta_1\leq5$&~\!\!&~\!\!\\ 
(3-3)&$[\![1,\boldsymbol{1}]\!]$&$3$&$\alpha_2=1,2$&$\alpha_2=3$,&$1\leq\beta_2\leq3$&$\alpha_2=3$,&$\beta_2=1$\\ 
(3-4)&$[\![1,\boldsymbol{2}]\!]$&$5$&$\alpha_2=1,2$&$3\leq\alpha_2\leq5$,&$2\leq\beta_2\leq5$&~&~\\ 
(3-5)&$[\![2,\boldsymbol{1}]\!]$&$4$&~&$2\leq\alpha_2\leq4$,&$1\leq\beta_2\leq4$&$\alpha_2=3$,&$\beta_2=1$\\ 
(3-6)&$[\![\boldsymbol{1},\boldsymbol{1}]\!]$&$3$&$\alpha_1=1,2$&$\alpha_1=3$,&$1\leq\beta_3\leq3$&~&\\ 
(3-7)&$[\![\boldsymbol{1},\boldsymbol{2}]\!]$&$4$&$\alpha_1=1,2$&$3\leq\alpha_1\leq4$,&$2\leq\beta_4\leq4$&~&\\ 
\hline
\multicolumn{8}{r}{C. A.= Considered Already}
\end{tabular}
\end{footnotesize}
\end{table}

\noindent\textbf{Case (3-1)} $(\alpha_1,\alpha_2)=(1,1)$.
It suffices to show that for $\beta_1\in\{1,2\}$, the equation
$$
3x^2+3y^2+\beta_1 z^2=3N+\beta_1
$$ 
has an integer solution $(x,y,z)\in\mathbb{Z}^3$ such that $z\not\equiv0\Mod3$.
Since $h(\langle\beta_1,3,3\rangle)=1$, one may easily show that a positive integer $m$ congruent to  $\beta_1$ modulo $3$ is represented by $\langle\beta_1,3,3\rangle$. 
This competes the proof. 
\end{proof}

\section{Quaternary mixed sums of generalized $4$- and $8$-gonal numbers}
In \cite{dick}, Dickson proved that there are exactly $54$ quaternary universal sums of squares. Note that any square of an integer is a  generalized $4$-gonal numbers. 
Recently, it was proved  in \cite{octagonal} and \cite{sun} that there are exactly $40$ quaternary universal sum of generalized $8$-gonal numbers. 
In this section, we determine all quaternary proper universal mixed sums of generalized $4$- and $8$-gonal numbers including the above $94$ quaternary universal sums.

Recall that a mixed sum $\Phi=[\![\alpha_1,\dots,\alpha_u,\boldsymbol{\beta_1},\dots,\boldsymbol{\beta_v}]\!]$ of generalized $4$-and $8$-goanl numbers is universal if and only if the equation 
\begin{equation}\label{core'}
\begin{array} {ll}
3\alpha_1x_1^2+\cdots+3\alpha_ux_u^2\!\!\!\!&+\beta_1y_1^2+\cdots+
\beta_vy_v^2\\
&=3N+\beta_1+\cdots+\beta_v \ \ \text{with} \  \  y_1\cdots y_v\not\equiv0\Mod3
\end{array}
\end{equation}
has an integer solution $(x_1,\dots,x_u,y_1,\dots,y_v)\in\mathbb{Z}^{u+v}$ for any nonnegative integer $N$.
Let 
$$
E=\{1\leq n< \beta_1+\cdots+\beta_v\mid n\equiv \beta_1+\cdots+\beta_v\Mod3\}.
$$ 
For each $n\in E$, let $\nu(n)$ be the positive integer such that
$$
4^{\nu(n)-1}\cdot n<\beta_1+\cdots+\beta_v\leq 4^{\nu(n)}\cdot n.
$$

\begin{lem}\label{reduction}
Under the same notations given above, assume that 
\begin{equation}\label{assumption1}
\begin{array} {ll}
3\alpha_1x_1^2+\cdots+3\alpha_ux_u^2\!\!\!\!&+\beta_1y_1^2+\cdots+
\beta_vy_v^2\\
&=4^{\nu(n)}\cdot n \ \ \text{with} \  \  y_1\cdots y_v\not\equiv0\Mod3
\end{array}
\end{equation}
has an integer solution for any $n\in E$.
Then the mixed sum $\Phi$ is universal if and only if the equation
$$
\begin{array} {ll}
3\alpha_1x_1^2+\cdots+3\alpha_ux_u^2\!\!\!\!&+\beta_1y_1^2+\cdots+
\beta_vy_v^2\\
&=3N+\beta_1+\cdots+\beta_v \ \ \text{with} \  \  y_1\cdots y_v\not\equiv0\Mod3
\end{array}
$$
has an integer solution for any nonnegative integer $N$ such that $3N+\beta_1+\cdots+\beta_v\not\equiv0\Mod4$.
\end{lem}

\begin{proof}
Note that ``only if" part is trivial. To prove ``if" part, let $N$ be any nonnegative integer. 
Let $s$ be a positive integer and $n$ be a positive integer not divisible by $4$ such that $3N+\beta_1+\cdots+\beta_v=4^s\cdot n$. 
If $n\geq\beta_1+\cdots\beta_v$, then there is a nonnegative integer $N_0$ such that $3N_0+\beta_1+\cdots+\beta_v=n$. Then 
Equation \eqref{core'} has an integer solution for this $N_0$ by assumption, and therefore Equation \eqref{core'} has also an integer solution for $N$. If $n<\beta_1+\cdots+\beta_v$, then $n\in E$ and $\nu(n)\leq s$ from the choice of $\nu(n)$. Since Equation \eqref{assumption1} has an integer solution for $4^{\nu(n)}\cdot n$, Equation \eqref{core'} has an integer solution for $N$. 
\end{proof}

\begin{thm}\label{quaternary}
There are exactly $547$ proper quaternary universal mixed sums of generalized $4$- and $8$-gonal numbers {\rm (}for the list of them, see Table 4.1{\rm)}.
\end{thm}

\begin{proof}
Let $u,v$ be nonnegative integers such that $u+v=4$. 
For positive integers $\alpha_1,\dots,\alpha_u$ and $\beta_1,\dots,\beta_v$, assume that a quaternary mixed sum
$$
\Phi=[\![\alpha_1,\dots,\alpha_u,\boldsymbol{\beta_1},\dots,\boldsymbol{\beta_v}]\!]
$$ 
  is proper universal. 
Then by theorem \ref{ternary}, there are nonnegative integers $u_1, v_1$ such that $u_1+v_1=3$ and the proper sum 
$$
\Phi'=[\![\alpha_1,\dots,\alpha_{u_1},\boldsymbol{\beta_1},\dots,\boldsymbol{\beta_{v_1}}]\!]
$$ 
of $\Phi$ is one of the sums in Table 4.1 which is not universal (see also Table 3.1).  
For each case, from the assumption that $\Phi$ is universal,  we know that at least one of $\alpha_{u_1+1}$ or $\beta_{v_1+1}$ is less than or equal to $t(\Phi')$ given in Table 4.1. Therefore, we have exactly $564$ candidates of proper quaternary universal mixed sums as in Table 4.1.

Now, we show that there are exactly $547$ proper quaternary  universal mixed sums among those candidates. The truants of the remaining $17$ quaternary sums are given in Table 5.1. If $u=0$ ($v=0$), then the universality of $\Phi$ can directly be proved by the  ``$60$-theorem" in \cite{octagonal} (``$15$-theorem" of \cite{b}, respectively). 
Therefore, it suffices to consider the cases when  $uv\not\eq0$.
Since the proofs  are quite similar to each other, we only provide, as representative cases,  the proofs of Cases
$$
 \text{(4-1)}, \ \text{(4-4)}, \  \text{(4-11)}, \  \text{(4-15)}, \ \text{(4-18)}, \ \text{(4-31)}, \ \text{and} \  \text{(4-35)}.
$$
In fact, the proof of Case (4-15) is one of the most complicate cases. 

To explain how to read Table 4.1, let us consider the case (4-11).  Note that $t([\![1,2,\boldsymbol{3}]\!])=10$. Hence all possible candidates of quaternary universal sums containing $[\![1,2,\boldsymbol{3}]\!]$ are 
$$
[\![1,2,\alpha_3,\boldsymbol{3}]\!] \ \ \text{for $2\le \alpha_3\le 10$} \ \ \text{or} \ \  [\![1,2,\boldsymbol{3},
\boldsymbol{\beta_2}]\!] \ \ \text{for $3\le \beta_2\le 10$}.
$$ 
Among these quaternary sums, the cases when $2\le \alpha_3\le5$  are already considered. For example, the case when $\alpha_3=2$ is considered in Case (4-5).   All the other quaternary sums except $\beta_2=3$ is, in fact,   universal. For the exceptional case, note that $t([\![1,2,\boldsymbol{3},\boldsymbol{3}]\!])=13$, which is considered in Section 5 in more detail.   

\begin{table}[h]
\centering
\begin{footnotesize}
\renewcommand{\arraystretch}{0.9}\renewcommand{\tabcolsep}{1mm}
\begin{tabular}{|c|c|c|ll|ll|ll|}\multicolumn{9}{c}{\textbf{Table 4.1.} Quaternary universal mixed sums}\\
\hline
\!\!Case&$\Phi'$&$t(\Phi')$\!\!&\multicolumn{2}{c|}{Considered Already} &\multicolumn{4}{c|}{Universal Cases} \\
\hline
(4-1)&$[\![1,1,1]\!]$&$7$&~&$\beta_1=1,2$&$1\leq\alpha_4\leq7$,&$3\leq\beta_1\leq7$&~&~\\
(4-2)&$[\![1,1,2]\!]$&$14$&~&$\beta_1=1,2$&$2\leq\alpha_4\leq14$,&$3\leq\beta_1\leq14$&~&~\\
(4-3)&$[\![1,1,3]\!]$&$6$&~&$\beta_1=1,2$&$3\leq\alpha_4\leq6$,&$3\leq\beta_1\leq6$&~&~\\
(4-4)&$[\![1,1,\boldsymbol{3}]\!]$&$6$&$1\leq\alpha_3\leq3$&~&$4\leq\alpha_3\leq6$,&$3\leq\beta_2\leq6$&~&~\\
(4-5)&$[\![1,2,2]\!]$&$7$&~&~&$2\leq\alpha_4\leq7$,&$1\leq\beta_1\leq7$&~&~\\
(4-6)&$[\![1,2,3]\!]$&$10$&~&$\beta_1=1$&$3\leq\alpha_4\leq10$,&$2\leq\beta_1\leq10$&~&~\\
(4-7)&$[\![1,2,4]\!]$&$14$&~&~&$4\leq\alpha_4\leq14$,&$1\leq\beta_1\leq14$&~&~\\
(4-8)&$[\![1,2,5]\!]$&$10$&~&~&$5\leq\alpha_4\leq10$,&$1\leq\beta_1\leq10$&$\alpha_4\not\eq5$,&$\beta_1\not\eq5$\\
(4-9)&$[\![1,2,\boldsymbol{1}]\!]$&$15$&$2\leq\alpha_3\leq6$,&$\beta_2=1$&$6\leq\alpha_3\leq15$,&$2\leq\beta_2\leq15$&~&$\beta_2\not\eq14$\\
(4-10)&$[\![1,2,\boldsymbol{2}]\!]$&$7$&$2\leq\alpha_3\leq5$&~&$6\leq\alpha_3\leq7$,&$2\leq\beta_2\leq7$&~&$\beta_2\not\eq7$\\
(4-11)&$[\![1,2,\boldsymbol{3}]\!]$&$10$&$2\leq\alpha_3\leq5$&~&$6\leq\alpha_3\leq10$,&$3\leq\beta_2\leq10$&~&$\beta_2\not\eq3$\\
(4-12)&$[\![1,2,\boldsymbol{4}]\!]$&$14$&$2\leq\alpha_3\leq5$&~&$6\leq\alpha_3\leq14$,&$4\leq\beta_2\leq14$&~&$\beta_2\not\eq14$\\
(4-13)&$[\![1,2,\boldsymbol{5}]\!]$&$10$&$2\leq\alpha_3\leq5$&~&$6\leq\alpha_3\leq10$,&$5\leq\beta_2\leq10$&$\alpha_3\not\eq10$,&$\beta_2\not\eq5,10$\\
(4-14)&$[\![1,\boldsymbol{1},\boldsymbol{2}]\!]$&$13$&$1\leq\alpha_2\leq3$&~&$4\leq\alpha_2\leq13$,&$2\leq\beta_3\leq13$&~&~\\
(4-15)&$[\![1,\boldsymbol{1},\boldsymbol{3}]\!]$&$18$&$1\leq\alpha_2\leq3$&~&$4\leq\alpha_2\leq18$,&$3\leq\beta_3\leq18$&~&~\\
(4-16)&$[\![1,3,\boldsymbol{2}]\!]$&$8$&~&~&$3\leq\alpha_3\leq8$,&$2\leq\beta_2\leq8$&~&~\\
(4-17)&$[\![1,4,\boldsymbol{2}]\!]$&$12$&~&~&$4\leq\alpha_3\leq12$,&$2\leq\beta_2\leq12$&~&~\\
(4-18)&$[\![1,5,\boldsymbol{2}]\!]$&$12$&~&~&$5\leq\alpha_3\leq12$,&$2\leq\beta_2\leq12$&~&~\\
(4-19)&$[\![1,\boldsymbol{2},\boldsymbol{2}]\!]$&$7$&$1\leq\alpha_2\leq5$&~&$6\leq\alpha_2\leq7$,&$2\leq\beta_3\leq7$&~&~\\
(4-20)&$[\![1,\boldsymbol{2},\boldsymbol{3}]\!]$&$8$&$1\leq\alpha_2\leq5$&~&$6\leq\alpha_2\leq8$,&$3\leq\beta_3\leq8$&~&~\\
(4-21)&$[\![1,\boldsymbol{2},\boldsymbol{4}]\!]$&$12$&$1\leq\alpha_2\leq5$&~&$6\leq\alpha_2\leq12$,&$4\leq\beta_3\leq12$&~&~\\
(4-22)&$[\![1,\boldsymbol{2},\boldsymbol{5}]\!]$&$12$&$1\leq\alpha_2\leq5$&~&$6\leq\alpha_2\leq12$,&$5\leq\beta_3\leq12$&~&~\\
(4-23)&$[\![2,2,\boldsymbol{1}]\!]$&$6$&$\alpha_3=3$,&$\beta_2=1$&$2\leq\alpha_3\leq6$,&$2\leq\beta_2\leq6$&~&~\\
(4-24)&$[\![2,4,\boldsymbol{1}]\!]$&$15$&~&$\beta_2=1$&$4\leq\alpha_3\leq15$,&$2\leq\beta_2\leq15$&~&$\beta_2\not\eq14$\\
(4-25)&$[\![2,\boldsymbol{1},\boldsymbol{2}]\!]$&$6$&$2\leq\alpha_2\leq4$&~&$5\leq\alpha_2\leq6$,&$2\leq\beta_3\leq6$&~&~\\
(4-26)&$[\![2,\boldsymbol{1},\boldsymbol{3}]\!]$&$14$&$2\leq\alpha_2\leq4$&~&$5\leq\alpha_2\leq14$,&$3\leq\beta_3\leq14$&~&$\beta_3\not\eq14$\\
(4-27)&$[\![2,\boldsymbol{1},\boldsymbol{4}]\!]$&$15$&$2\leq\alpha_2\leq4$&~&$5\leq\alpha_2\leq15$,&$4\leq\beta_3\leq15$&~&~\\
(4-28)&$[\![3,\boldsymbol{1},\boldsymbol{1}]\!]$&$7$&~&~&$3\leq\alpha_2\leq7$,&$1\leq\beta_3\leq7$&~&$\beta_3\not\eq7$\\
(4-29)&$[\![\boldsymbol{1},\boldsymbol{1},\boldsymbol{1}]\!]$&$4$&$1\leq\alpha_1\leq3$&~&$\alpha_1=4$,&$1\leq\beta_4\leq4$&~&~\\
(4-30)&$[\![\boldsymbol{1},\boldsymbol{1},\boldsymbol{2}]\!]$&$14$&$1\leq\alpha_1\leq3$&~&$4\leq\alpha_1\leq14$,&$2\leq\beta_4\leq14$&~&$\beta_4\not\eq14$\\
(4-31)&$[\![\boldsymbol{1},\boldsymbol{1},\boldsymbol{3}]\!]$&$7$&$1\leq\alpha_1\leq3$&~&$4\leq\alpha_1\leq7$,&$3\leq\beta_4\leq7$&$\alpha_1\not\eq7$,&$\beta_4\not\eq4,7$\\
(4-32)&$[\![3,\boldsymbol{1},\boldsymbol{2}]\!]$&$9$&~&~&$3\leq\alpha_2\leq9$,&$2\leq\beta_3\leq9$&~&~\\
(4-33)&$[\![4,\boldsymbol{1},\boldsymbol{2}]\!]$&$13$&~&~&$4\leq\alpha_2\leq13$,&$2\leq\beta_3\leq13$&~&~\\
(4-34)&$[\![\boldsymbol{1},\boldsymbol{2},\boldsymbol{2}]\!]$&$6$&$1\leq\alpha_1\leq4$&~&$5\leq\alpha_1\leq6$,&$2\leq\beta_4\leq6$&~&~\\
(4-35)&$[\![\boldsymbol{1},\boldsymbol{2},\boldsymbol{3}]\!]$&$9$&$1\leq\alpha_1\leq4$&~&$5\leq\alpha_1\leq9$,&$3\leq\beta_4\leq9$&~&$\beta_4\not\eq3$\\
(4-36)&$[\![\boldsymbol{1},\boldsymbol{2},\boldsymbol{4}]\!]$&$13$&$1\leq\alpha_1\leq4$&~&$5\leq\alpha_1\leq13$,&$4\leq\beta_4\leq13$&~&~\\
\hline 
\end{tabular}
\end{footnotesize}
\end{table}

\vskip0.5pc
\noindent\textbf{Case (4-1)} $(\alpha_1,\alpha_2,\alpha_3)=(1,1,1)$.
 It is enough to show that for any $\beta_1$ such that $3\leq\beta_1\leq7$, the equation
$$
3x^2+3y^2+3z^2+\beta_1 t^2=3N+\beta_1
$$ 
has an integer solution $(x,y,z,t)\in\mathbb{Z}^4$ such that $t\not\equiv0 \Mod 3$ for any nonnegative integer $N$.

Since the proofs are quite similar to each other, we only provide the proof of the case $\beta_1=3$.
By Lemma \ref{reduction}, we may assume that $3N+3\not\equiv0\Mod4$.
If $1\leq N\leq 15$, then one may directly check that  the equation
$$
x^2+y^2+z^2+t^2=N+1
$$ 
has an integer solution $(x,y,z,t)\in\mathbb{Z}^4$ such that $t\not\equiv0\Mod 3$. 
Therefore, we assume that $N\geq 16$. 
Note that an integer is represented by $\langle1,1,1\rangle$ if and only if it is not of the form $4^{l}(8k+7)$ for any nonnegative integers $l,k$.
Hence one may easily show that there is an integer $d\in\{1,4\}$ such that $N+1-d^2>0$, and furthermore $N+1-d^2$ is represented by $\langle1,1,1\rangle$. 
Therefore, the equation
$$
x^2+y^2+z^2+t^2=N+1
$$ 
has an integer solution $(x,y,z,t)\in\mathbb{Z}^4$ such that $t\not\equiv0 \Mod 3$ for any nonnegative integer $N$.
This completes the proof.

\vskip0.5pc
\noindent\textbf{Case (4-4)} $(\alpha_1,\alpha_2,\beta_1)=(1,1,3)$. 
 It is enough to show that for any $\alpha_3(\beta_2)$ such that $4\leq\alpha_3\leq6$ $(3\leq\beta_2\leq6)$, the equation
$$
x^2+y^2+\alpha_3 z^2+t^2=N+1~(3x^2+3y^2+3z^2+\beta_2 t^2=3N+3+\beta_2)
$$ 
has an integer solution $(x,y,z,t)\in\mathbb{Z}^4$ such that $t\not\equiv0\Mod 3$$(zt\not\equiv0\Mod3$, respectively$)$ for any nonnegative integer $N$. 

Since the proofs are quite similar to each other, we only provide the proof of the case $\alpha_3=4$.
By Lemma \ref{reduction}, we may assume that $N+1\not\equiv0\Mod4$.
If $N\leq3$, then one may directly check that 
$$
x^2+y^2+4z^2+t^2=N+1
$$ 
has an integer solution $(x,y,z,t)\in\mathbb{Z}^4$ such that $t\not\equiv0\Mod 3$. 
Therefore, we assume that $N\geq 4$. 
Note that a positive integer $m$ is represented by $\langle1,1,1\rangle$ if and only if it is not of the form $4^l(8k+7)$ for any nonnegative integers $l,k$.  Furthermore, if $m$ is represented by $\langle 1,1,1\rangle$, then there is an integer solution $(x,y,t) \in \mathbb Z^3$ such that $x^2+y^2+t^2=m$ and $\text{gcd}(x,y,t,p)=1$ for any odd prime $p$. This can be proved by using a slightly modified version of 102:5 of \cite{om}. 
   Hence, there is an integer $c\in\{0,1\}$ such that the diophantine equation
$$
x^2+y^2+t^2=N+1-4c^2 \quad \text{with} \quad t \not \equiv 0 \Mod 3
$$
has an integer solution $(x,y,t) \in \mathbb Z^3$.
This completes the proof of Case (4-4).  

\vskip0.5pc
\noindent\textbf{Case (4-11)} $(\alpha_1,\alpha_2,\beta_1)=(1,2,3)$.
It suffices to show that for any $\alpha_3(\beta_2)$ such that $6\leq\alpha_3\leq10$ $(4\leq\beta_2\leq10)$, the equation 
$$
x^2+2y^2+\alpha_3 z^2+t^2=N+1~(3x^2+6y^2+3z^2+\beta_2t^2=3N+3+\beta_2)
$$ 
has an integer solution $(x,y,z,t)\in\mathbb{Z}^4$ such that $t\not\equiv0\Mod3$$(zt\not\equiv0\Mod3$, respectively) for any nonnegative integer $N$. 

Assume $\alpha_3=7$. 
By Lemma \ref{reduction}, we may assume that $N+1\not\equiv0\Mod4$.
If $1\leq N\leq111$ then one may directly check that such an integer solution always exists.
Therefore, we may assume that $N\geq 112$. 
Note that an integer is represented by $\langle1,1,2\rangle$ if and only if it is not of the form $2^{2l+1}(8k+7)$ for any nonnegative integers $l,k$. 
Hence, one may easily show that there is an integer $c\in\{0,1,2,3,4\}$ such that $N+1-7c^2>0$, $N+1-7c^2\equiv0,1\Mod3$, and furthermore, $N+1-7c^2$ is represented by $\langle1,1,2\rangle$. 
Therefore, the equation 
$$
x^2+2y^2+t^2=N+1-7c^2
$$ 
has an integer solution $(x,y,t)=(a,b,d)\in\mathbb{Z}^3$ such that $a\not\equiv0\Mod3$ or $d\not\equiv0\Mod3$ or $a\equiv b\equiv d\equiv 0\Mod 3$. 
 Assume that $a\equiv b\equiv d\equiv 0\Mod 3$. 
We may assume that $a^2+2b^2\not\eq0$. 
By Theorem 9 of \cite{Jones} (see also \cite{Jagy}, and for more generalization, see \cite{oh}), there are integers $e,f$ such that $e^2+2f^2=a^2+2b^2$ and $ef\not\equiv0\Mod3$. 
Therefore, the equation 
$$
x^2+2y^2+7z^2+t^2=N+1
$$ 
has an integer solution $(x,y,z,t)\in\mathbb{Z}^4$ such that $t\not\equiv0\Mod3$ for any nonnegative integer $N$. 

Assume $\beta_2=9$. It suffices to show that the equation 
$$
x^2+2y^2+z^2+3t^2=N+4
$$ 
has an integer solution $(x,y,z,t)\in\mathbb{Z}^4$ such that $zt\not\equiv0\Mod3$. 
By Lemma \ref{reduction}, we may assume that $N+4\not\equiv0\Mod4$.
If $N\leq 165$, then one may directly check that 
$$
x^2+2y^2+z^2+3t^2=N+4
$$ 
has an integer solution $(x,y,z,t)\in\mathbb{Z}^4$ such that $zt\not\equiv0\Mod3$.
Therefore, we assume $N\geq166$. If $N+4\equiv0,1\Mod3$, then the proof is quite similar to the proof of the case $\alpha_3=7$. 
Assume $N+4\equiv2\Mod3$. 
Since $h(\langle1,2,3\rangle)=1$, one may easily show that a positive integer which is not of the form $2^{2l+1}(8k+5)$ for any nonnegative integers $l,k$ is represented by $\langle1,2,3\rangle$.
Hence, there is an integer $c\in\{1,2,4,5,7,13\}$ such that  $N+4-c^2>0$, $N+4-c^2\equiv4\Mod6$, and furthermore, $N+4-c^2$ is represented by $\langle1,2,3\rangle$. 
Therefore, there are integers $a,b,d$ such that $a^2+2b^2+3d^2=N+4-c^2$ with $a\equiv d\Mod2$.  If $d\not\equiv0\Mod3$, then we are done.
Assume $d\equiv0\Mod3$. Note that $\left({\frac{a+3d}{2}}\right)^2+3\left({\frac{a-d}{2}}\right)^2=a^2+3d^2$ 
and $\left(\frac{a+3d}{2}\right)\cdot\left(\frac{a-d}{2}\right)\not\equiv0\Mod3$. 
Therefore, the equation 
$$
x^2+2y^2+z^2+3t^2=N+4
$$ 
has an integer solution $(x,y,z,t)\in\mathbb{Z}^4$ such that $zt\not\equiv0\Mod3$. 

The proofs of all the other cases are quite similar to those of the case $\alpha_3=7$ or Case (4-1).

Assume that the proper sum $\Phi'$ of $\Phi$ is  one of the sums in Table 4.2. Then the proofs are quite similar to those of Cases (4-1) or (4-11).

\begin{table}[h]
\centering
\begin{footnotesize}
\renewcommand{\arraystretch}{1}\renewcommand{\tabcolsep}{3mm}
\begin{tabular}{|c|c||c|c||c|c|}\multicolumn{6}{c}{\textbf{Table 4.2.} All cases whose proofs are similar to (4-1) or (4-11)}\\
\hline
Case&$\Phi'$&Case&$\Phi'$&Case&$\Phi'$\\
\hline
(4-2)&$[\![1,1,2]\!]$&(4-10)&$[\![1,2,\boldsymbol{2}]\!]$&(4-27)&$[\![2,\boldsymbol{1},\boldsymbol{4}]\!]$\\   
(4-3)&$[\![1,1,3]\!]$&(4-14)&$[\![1,\boldsymbol{1},\boldsymbol{2}]\!]$&(4-28)&$[\![3,\boldsymbol{1},\boldsymbol{1}]\!]$\\
(4-5)&$[\![1,2,2]\!]$&(4-16)&$[\![1,3,\boldsymbol{2}]\!]$&(4-29)&$[\![\boldsymbol{1},\boldsymbol{1},\boldsymbol{1}]\!]$\\
(4-6)&$[\![1,2,3]\!]$&(4-17)&$[\![1,4,\boldsymbol{2}]\!]$&(4-30)&$[\![\boldsymbol{1},\boldsymbol{1},\boldsymbol{2}]\!]$\\   
(4-7)&$[\![1,2,4]\!]$&(4-19)&$[\![1,\boldsymbol{2},\boldsymbol{2}]\!]$&(4-34)&$[\![\boldsymbol{1},\boldsymbol{2},\boldsymbol{2}]\!]$\\
(4-8)&$[\![1,2,5]\!]$&(4-23)&$[\![2,2,\boldsymbol{1}]\!]$&(4-36)&$[\![\boldsymbol{1},\boldsymbol{2},\boldsymbol{4}]\!]$\\    
(4-9)&$[\![1,2,\boldsymbol{1}]\!]$&(4-25)&$[\![2,\boldsymbol{1},\boldsymbol{2}]\!]$&&\\
\hline
\end{tabular}
\end{footnotesize}
\end{table}

Note that for each case, the proof uses the set of integers that are represented by a particular ternary quadratic form having class number one, which are easily computable by Hasse principle.    
\vskip0.5pc

One of most important ingredients in the proofs of Cases (4-1), (4-4) and (4-11) is the existence of a ternary quadratic form having class number $1$. We crucially used the fact that the set of integers represented by such a ternary quadratic form is completely known. In the most of remaining cases, the quaternary quadratic form induced by a quaternary mixed sum of $4$- and $8$-gonal numbers under consideration does not have any ternary  quadratic subform having class number one.  Hence we use Theorem  \ref{pme} to compute the set of integers  in some arithmetic progression  that are represented by a ternary quadratic subform.  

\vskip 0.5pc
\noindent\textbf{Case (4-15)} $(\alpha_1,\beta_1,\beta_2)=(1,1,3)$. It suffices to show that for any $\alpha_2(\beta_3)$ such that $4\leq\alpha_2\leq18$ $(3\leq\beta_3\leq18)$, the equation 
$$
3x^2+3\alpha_2 y^2+z^2+3t^2=3N+4~(3x^2+y^2+3z^2+\beta_3 t^2=3N+4+\beta_3)
$$ 
has an integer solution $(x,y,z,t)\in\mathbb{Z}^4$ such that $zt\not\equiv0\Mod3 (yzt\not\equiv0\Mod3$, respectively$)$ for any nonnegative integer $N$. 

First, assume $\alpha_2=4$. 
If $N\leq34$, then one may directly check that such an integer solution always exists. Therefore, we may assume $N\geq 35$.  Define 
$$
f(x,z,t)=3x^2+(3z+t)^2+3t^2\quad \text{and}\quad g(x,z,t)=27x^2+z^2+27t^2.
$$
Then one may directly compute that 
\begin{equation}\label{(18)}
g\prec_{11,r}f
\end{equation}
 for any remainder $r$ modulo $11$, where $r$ is $0$ or a quadratic non-residue modulo $11$.

Now, choose an integer $d\in\{0,1,2,3\}$ such that  $3N+4-12b^2>0$ and  $3N+4-12b^2$ is divisible by $11$ or a quadratic non-residue modulo $11$. 
Since $h(\langle 1,3,3\rangle)=1$,  one may easily check that
$$
3x^2+z^2+3t^2=3N+4-12b^2
$$
has an integer solution $(x,z,t) \in \mathbb Z^3$. 
Note that $z$ is not divisible by $3$.  
If either  $x$ or $t$ is not divisible by $3$, then we are done.  If both $x$ and $t$ are divisible by $3$, that is, 
 $3N+4-12b^2$ is represented by $g$,  then by \eqref{(18)},  it is represented by $f$.  
Therefore,  the equation 
$$
3x^2+12y^2+z^2+3t^2=3N+4
$$ 
has an integer solution $(x,y,z,t)\in\mathbb{Z}^4$ such that $zt\not\equiv0\Mod3$. 

In fact, if $\alpha_2\not\eq11$ or $\beta_3\not\eq11$, then the proofs of the cases are quite similar to the above.

Assume $\alpha_2=11$.
 If $N\leq9$, then one may directly check that the equation  
$$
3x^2+33y^2+z^2+3t^2=3N+4
$$ 
has an integer solution $(x,y,z,t)\in\mathbb{Z}^4$ such that $zt\not\equiv0\Mod3$. 
Therefore, we may assume $N\geq10$.
Note that the genus of $g(x,z,t)=27x^2+z^2+27t^2$ consists of 
$$
M_g=\langle1,27,27,\rangle,\quad M_2=\begin{pmatrix}4&1&0\\1&7&0\\0&0&27\end{pmatrix},\quad\text{and}\quad M_3=\begin{pmatrix}7&-3&2\\-3&9&3\\2&3&16\end{pmatrix}.
$$
By Minkowski-Siegel formula (for details, see \cite{ya}), we have
\begin{equation}\label{(18)m}
r(3N+4,\langle1,3,3\rangle)-r(3N+4,\langle1,27,27\rangle)=4\cdot r(3N+4,M_2)+4\cdot r(3N+4,M_3). 
\end{equation}
One may easily check that 
$$
M_g\prec_{4,0}M_2 \quad\text{and}\quad M_g\prec_{4,3}M_2.
$$
This implies that any integer congruent to $4$ or $7$ modulo $12$ is represented by $M_2$ or $M_3$.  
 First, assume $3N+4\equiv0,1,3\Mod4$. 
Then there is an integer $b\in\{0,1\}$ such that $3N+4-33b^2>0$ and  $3N+4-33b^2$ is congruent to $4$ or $7$ modulo $12$. 
Therefore by \eqref{(18)m}, we have
$$
r(3N+4-33b^2,\langle1,3,3\rangle)-r(3N+4-33b^2,\langle1,27,27\rangle)>0.
$$
Consequently,  the equation  
$$
3x^2+33y^2+z^2+3t^2=3N+4
$$ 
has an integer solution $(x,y,z,t)\in\mathbb{Z}^4$ such that $zt\not\equiv0\Mod3$.
 Now, assume $3N+4\equiv2\Mod4$.
 Define $h(x,y,z)=2x^2+2xy+17y^2+3z^2$. The genus of $h$ consists of 
$$
M_{h}=\begin{pmatrix}2&1\\1&17\end{pmatrix}\perp\langle3\rangle\quad\text{and}\quad M_4=\begin{pmatrix}5&2&1\\2&5&1\\1&1&5\end{pmatrix}.
$$
One may easily check that any positive integer $m$ congruent to $2$ modulo $3$ is represented by $M_{h}$ or $M_4$ by 102:5 of \cite{om}, for it is represented by $M_{h}$ over $\mathbb{Z}_p$ for any prime $p$.
One may easily check that 
$$
M_4\prec_{4,0}M_{h}\quad\text{and}\quad M_4\prec_{4,1}M_{h}.
$$ 
Now, choose an integer $d\in\{1,2\}$ such that $\frac12(3N+4)-3d^2>0$ and $\frac12(3N+4)-3d^2$ is congruent to $5$ or $8$ modulo $12$. Then  $\frac12(3N+4)-3d^2$ is represented by $M_{h}$. 
Therefore there are integers $a,b,c$ such that 
$$
(2a+b)^2+33b^2+3(c+d)^2+3(c-d)^2=3N+4.
$$ 
Since $d\not\equiv0\Mod3$, $c+d\not\equiv0\Mod3$ or $c-d\not\equiv0\Mod3$. Consequently, the equation  
$$
3x^2+33y^2+z^2+3t^2=3N+4
$$ 
has an integer solution $(x,y,z,t)\in\mathbb{Z}^4$ such that $zt\not\equiv0\Mod3$.

Assume $\beta_3=11$.  First, we assume that $N\geq88$.
Under the same notations as above, one may directly compute that 
$$
B_{M_3}(M_g,5,1)=\{\pm(1,0,0)\}\quad\text{and}\quad B_{M_3}(M_g,5,4)=\{\pm(2,0,0)\}.
$$ 
In each case, if we define $T=\begin{pmatrix}5&0&0\\0&4&-3\\0&3&4\end{pmatrix}$, then one may easily show that it satisfies all conditions in Theorem \ref{pme}. Note that $\pm (1,0,0)$ are the only integral primitive eigenvectors of $T$. 
Therefore, we have
\begin{equation} \label{core3}
S_{5,r}\cap Q(M_g)\setminus\{s^2\mid s\in\mathbb Z\}\subset Q(M_3),
\end{equation}
for any remainder $r$ modulo $5$, where $r$ is a quadratic residue modulo $5$.
Now, choose an integer $d\in\{1,2,5\}$ such that $I(N,d):=3N+15-11d^2\geq2$ and $I(N,d)\equiv1,4\Mod5$. 
We show that  $I(N,d)$ is also represented by $M_2$ or $M_3$. 
If $I(N,d)$ is not a square of an integer, then it is represented by $M_2$ or $M_3$ by \eqref{core3}. Assume that 
$I(N,d)$ is a square of an integer.
Since $4$ is represented by $M_2$, any square divisible by $4$ is represented by $M_2$. Since $9$ is represented by $M_3$, we may assume that $I(N,d)$ is an odd integer not divisible by $3$.      
For any prime $p$ dividing $I(N,d)$,  $p^2$ is represented by $M_2$ or $M_3$ by Lemma 2.4 of \cite{poly}, for both $M_2$ and $M_3$ are contained in the spinor genus of $M_g$. 
Therefore
$$
r(I(N,d),\langle1,3,3\rangle)-r(I(N,d),\langle1,27,27\rangle)>0,
$$
 by \eqref{(18)m}.
Hence, the equation 
$$
3x^2+y^2+3z^2+11t^2=3N+15
$$ 
has an integer solution $(x,y,z,t)\in\mathbb{Z}^4$ such that $yzt\not\equiv0\Mod3$. If $N \le 87$, then one may directly check that such an integer solution exists. 

Recall that we use \eqref{good} and Theorem \ref{pme} to show that integers in an arithmetic progressions are represented by the ternary quadratic forms.
In Case (4-18), as a representative case, we explain how our method works in detail. Since everything is quite similar to this for all the other cases, we only provide all parameters needed for computations in Table 4.3.

\vskip 0.5pc
\noindent\textbf{Case (4-18)} $(\alpha_1,\alpha_2,\beta_1)=(1,5,2)$.
It suffices to show that for any $\alpha_3(\beta_2)$ such that $5\leq\alpha_3\leq12$ $(2\leq\beta_2\leq12)$, the equation 
$$
3x^2+15y^2+3\alpha_3 z^2+2t^2=3N+2~(3x^2+15y^2+2z^2+\beta_2 t^2=3N+2+\beta_2)
$$ 
has an integer solution $(x,y,z,t)\in\mathbb{Z}^4$ such that $t\not\equiv0\Mod3(zt\not\equiv0\Mod3$, respectively$)$ for any nonnegative integer $N$.

First, assume that $\alpha_3=5$.
By Lemma \ref{reduction}, we may assume that $3N+2\not\equiv0\Mod4$.
If $N \leq 124$, then  one may directly check that 
$$
3x^2+15y^2+15z^2+2t^2=3N+2
$$
has an integer solution $(x,y,z,t)\in\mathbb{Z}^4$ such that $t\not\equiv0\Mod3$. Therefore we may assume that $N\geq125$.
Note that the genus of $f(x,y,t)=3x^2+15y^2+2t^2$ 
consists of 
$$
M_f=\langle2,3,15\rangle\quad\text{and}\quad M_2=\begin{pmatrix}2&1&-1\\1&5&1\\-1&1&11\end{pmatrix}.
$$
One may easily show that any positive integer $m$ congruent to $2$ modulo $3$ that is not of the form $2^{2l+1}(8k+3)$ for any nonnegative integers $l,k$ is represented by $M_f$ or $M_2$ by 102:5 of \cite{om}.    Now, one may easily check that  
$$
M_2\prec_{2,0}M_f.
$$ 
Therefore, any positive integer $m$ congruent to $2$ modulo $6$  that is not of the form $2^{2l+1}(8k+3)$ is represented by $M_f$. 
Hence we may choose an integer $c\in\{1,2,3,4,5\}$ such that $3N+2-15c^2>0$, $3N+2-15c^2\equiv2\Mod6$, and furthermore, $3N+2-15c^2$  is represented by $M_f$. 
Therefore, the equation 
$$
3x^2+15y^2+15z^2+2t^2=3N+2
$$
has an integer solution $(x,y,z,t)\in\mathbb{Z}^4$ such that $t\not\equiv0\Mod3$.

If either $\alpha_3$ or $\beta_2$ is odd, then each proof is quite similar to that of the case $\alpha_3=5$.

Assume $\alpha_3=8$. 
Note that the genus of $f(x,z,t)=3x^2+24z^2+2t^2$ consists of 
$$
M_{f}=\langle2,3,24\rangle\quad\text{and}\quad M_2=\begin{pmatrix}5&1\\1&5\end{pmatrix}\perp\langle6\rangle. 
$$
One may easily show that a positive integer $m$ which is congruent to $2$ modulo $3$ and not of the form $4^{l}(8k+1)$ and $4^{l}(8k+7)$ for any nonnegative integers $l,k$ is represented by $M_{f}$ or $M_2$ by 102:5 in \cite{om} for it is represented by $M_{f}$ over $\mathbb{Z}_p$ for any prime $p$.
Note that 
$$
B_{f}(M_2,3,2)=\{\pm(1,2,0)\}.
$$ 
If we define $T=\begin{pmatrix}2&-1&-2\\-1&2&-2\\2&2&1\end{pmatrix}$, 
then one may easily show that it satisfies all conditions in Theorem \ref{pme}. 
Note that $\pm(1,-1,0)$ are the only integral primitive eigenvectors of $T$. Therefore we have 
$$
S_{3,2}\cap Q(M_2)\setminus\{8s^2\mid s\in\mathbb Z\}\subset Q(M_{f}).
$$
Since $8$ is represented by $M_{f}$, $8s^2$ is represented by $M_{f}$ for any $s\in\mathbb Z$. 
Therefore, a positive integer $m$ which is congruent to $2$ modulo $3$ and not of the form $4^{l}(8k+1)$ and $4^{l}(8k+7)$ is represented by $M_{f}$. 
The rest of the proof is quite similar to that of $\alpha_3=5$.

Assume $\alpha_3=12$.
Note that the genus of $f(x,z,t)=3x^2+36z^2+2t^2$ consists of
$$
M_{f}=\langle2,3,36\rangle\quad\text{and}\quad M_2=\langle2,9,12\rangle.
$$
Note that 
$$
M_2\prec_{2,0}M_{f}.
$$ 
One may easily show that every positive integer congruent to $2$ modulo $6$ that is not of the form $2^{2l+1}(8k+7)$ for any nonnegative integers $l,k$ is represented by $M_f$. 
The rest of the proof is quite similar when $\alpha_3=5$.

Assume $\beta_2=10$.
By Lemma \ref{reduction}, we may assume that $3N+12\not\equiv0\Mod4$.
If $N\leq1276$, then one may directly check that the equation 
$$
3x^2+15y^2+2z^2+10t^2=3N+12
$$ 
has an integer solution $(x,y,z,t)\in\mathbb{Z}^4$ such that $zt\not\equiv0\Mod3$. 
Therefore, assume $N\geq1277$.
Note that the genus of $f(x,z,t)=\frac13(3x^2+2(3z+t)^2+10t^2)=x^2+6z^2+4t^2+4zt$ consists of 
$$
M_{f}=\langle1\rangle\perp\begin{pmatrix}4&2\\2&6\end{pmatrix} \quad\text{and}\quad M_2=\langle2,2,5\rangle.
$$
One may easily show that a positive integer $m$ which is not of the form $4^{l}(8k+3)$ for any nonnegative integers $l,k$ is represented by $M_{f}$ or $M_2$ by 102:5 in \cite{om} for it is represented by $M_{f}$ over $\mathbb{Z}_p$ for any prime $p$. 
Note that 
$$
M_2\prec_{3,0}M_{f}\quad\text{and}\quad  M_2\prec_{3,1}M_{f}.
$$ 
Therefore, a positive integer $m$ that is congruent to $0$ or $1$ modulo $3$ and not of the form $4^l(8k+3)$ is represented by $M_f$.
Hence there is an integer $b\in\{0,1,3,4\}$ such that both $N+4-5b^2$ and $N+4-5(b+12)^2$ are positive and represented by $M_{f}$. 
For all possible cases when both $N+4-5b^2$ and $N+4-5(b+12)^2$ are squares of integers, one may easily check that the equation 
$$
3x^2+15y^2+2z^2+10t^2=3N+12
$$ 
has an integer solution $(x,y,z,t)\in\mathbb{Z}^4$ such that $zt\not\equiv0\Mod3$. 
Therefore we may assume that $3x^2+2z^2+10t^2=3N+12-15b^2$ has an integer solution $(x,z,t)=(a,c,d)\in\mathbb{Z}^3$ such that  $2c^2+10d^2\not\eq0$.
If $c\equiv d\equiv0\Mod3$, then there are integers $e,f$ such that $e^2+5f^2=c^2+5d^2$ and $ef\not\equiv0\Mod3$, by Theorem 9 of \cite{Jones}. This completes the proof. 

Assume $\beta_2=12$.
By Lemma \ref{reduction}, we may assume that $3N+14\not\equiv0\Mod4$.
If $N\leq1800$, then one may directly check that the equation
$$
3x^2+15y^2+2z^2+12t^2=3N+14
$$
has an integer solution $(x,y,z,t)\in\mathbb{Z}^4$ such that $zt\not\equiv0\Mod3$. 
Therefore, assume $N\geq1801$.
Define 
$$
f(x,z,t)=12x^2+2(3z+t)^2+12t^2 ~\text{and}~  g(x,z,t)=108x^2+2z^2+108t^2.
$$
Note that 
\begin{equation}\label{(21)}
g\prec_{7,r}f_4,
\end{equation}
for any remainder $r$ modulo $7$, where $r$ is $0$ or a quadratic non-residue modulo $7$.
Since $h(\langle1,6,6\rangle)=1$, one may easily show that a positive integer $m$ which is congruent to $1$ modulo $3$ and not of the form $4^{l}(8k+3)$ for any nonnegative integers $l,k$ is represented by $\langle1,6,6\rangle$.
Hence there is an integer $0\leq b\leq19$ such that $3N+14-15b^2$ is positive, $3N+14-15b^2$ is divisible by 7 or a quadratic non-residue modulo $7$, and furthermore, $3N+14-15b^2$ is represented by $\langle2,12,12\rangle$. 
If $3N+14-15b^2$ is not represented by $g$, then we are done. 
Assume that $3N+14-15b^2$ is represented by $g$. By Equation \eqref{(21)}, $3N+14-15b^2$ is represented by $f$. 
Therefore the equation 
$$
3(2x)^2+15y^2+2(3z+t)^2+12t^2=3N+14
$$ 
has an integer solution. This completes the proof.

The proofs of all remaining cases are quite similar to those of Cases (4-1) or (4-11). 
This completes the proof of Case (4-18).

As noted earlier, since every proof of the case in the Table 4.3 is quite similar to that of Case (4-18), we only provide all parameters needed for the computations. 

\begin{table}[h]
\centering
\begin{footnotesize}
\renewcommand{\arraystretch}{0.9}\renewcommand{\tabcolsep}{0.3mm}
\begin{tabular}{|c|c|c|c|c|c|l|}\multicolumn{7}{c}{\textbf{Table 4.3.} Some data for quaternary universal  mixed sums}\\
\hline

Case&$M_f$&$M_2$&$(d,a)$&$B_f(M_2,d,a)$&$T$\!\!&\multicolumn{1}{c|}{$f$ represents $m$}\\
\hline
(4-12)&$\langle3,4,6\rangle$&$\langle1,6,12\rangle$&$(3,1)$&$\{\pm(1,0,0)\}$&$\begin{pmatrix}3&0&0\\0&1&-4\\0&2&1\end{pmatrix}$&$\setlength\arraycolsep{2pt}\begin{array}{l}m\geq2,\\ m\equiv1\Mod3,\\ m\not\eq2^{2l+1}(8k+7)\end{array}$\\
\hline

(4-13)&$\langle3,5,6\rangle$&$\begin{pmatrix}2&0&0\\0&6&3\\0&3&9\end{pmatrix}$&$(3,2)$&$\{\pm(1,0,0)\}$&$\begin{pmatrix}3&0&0\\0&3&3\\0&-2&1\end{pmatrix}$&$\setlength\arraycolsep{2pt}\begin{array}{l}m\equiv3\Mod2,\\ m\not\eq5^{2l+1}(5k+2),\\ m\not\eq5^{2l+1}(5k+3), \\m\not\eq2\cdot25^l\end{array}$\\
\hline

$\setlength\arraycolsep{2pt}\begin{array}{c}\text{(4-20)}\\\beta_3\!\not\eq\!3,4\end{array}$&$\begin{pmatrix}3&0&0\\0&5&1\\0&1&11\end{pmatrix}$&$\langle2,3,27\rangle$&$\setlength\arraycolsep{2pt}\begin{array}{c}(13,2)\\(13,5)\\(13,6)\\(13,7)\\(13,8)\\ (13,11)\end{array}$&$\setlength\arraycolsep{2pt}\begin{array}{l}\{\pm(1,0,0)\}\\\{\pm(3,0,0)\}\\\{\pm(4,0,0)\}\\\{\pm(6,0,0)\}\\\{\pm(2,0,0)\}\\\{\pm(5,0,0)\}\end{array}$&$\begin{pmatrix}13&0&0\\0&5&-36\\0&4&5\end{pmatrix}$&$\setlength\arraycolsep{2pt}
\begin{array}{l}m\geq3,\\m\equiv2\Mod3,\\m\equiv r_1\Mod{13}\end{array}$\\
\hline

(4-21)&$\langle2,3,4\rangle$&$\langle1,2,12\rangle$&$(3,0)$&$\emptyset$&~&$\setlength\arraycolsep{2pt}\begin{array}{l}m\equiv0\Mod3,\\m\not\eq2^{2l+1}(8k+5)\end{array}$\\
\hline

$\setlength\arraycolsep{2pt}\begin{array}{c}\text{(4-22)}\\\alpha_2\not\eq7\\\beta_3\not\eq7\end{array}$&$\langle2,3,5\rangle$&$\langle1,1,30\rangle$&$(7,r_2)$&$\emptyset$&~&$\setlength\arraycolsep{2pt}\begin{array}{l}m\equiv1\Mod3,\\m\equiv r_2\Mod7\end{array}$\\
\hline

\multirow{4}{*}{\!\!$\setlength\arraycolsep{2pt}\begin{array}{c}\text{(4-22)}\\\alpha_2=7\end{array}$}&$\langle3,5,21\rangle$&$\begin{pmatrix}5&0&0\\0&6&3\\0&3&12\end{pmatrix}$&$(8,r_3)$&$\emptyset$&~&\multirow{2}{*}{
$\setlength\arraycolsep{2pt}\begin{array}{l}m\equiv2\Mod3,\\m\not\equiv0\Mod5,\\m\not\equiv0\Mod7,\\m\equiv r_3\Mod8\end{array}$}\\
\cline{2-5}
&$\langle3,5,21\rangle$&$\begin{pmatrix}2&1&0\\1&8&3\\0&0&21\end{pmatrix}$&$(8,r_3)$&$\emptyset$& &\\
\hline

\multirow{2}{*}{$\setlength\arraycolsep{2pt}\begin{array}{c}\text{(4-22)}\\\beta_3=7\end{array}$}&$\langle2,3,7\rangle$&$\begin{pmatrix}2&1&1\\1&3&0\\1&0&9\end{pmatrix}$&$(3,0)$&$\emptyset$&~&\multirow{2}{*}{
$\setlength\arraycolsep{2pt}\begin{array}{l}m\equiv6\Mod9,\\m\not\eq2^{2l+1}(8k+3)\end{array}$}\\
\cline{2-5}
&$\langle2,3,7\rangle$&$\langle1,3,14\rangle$&$(3,0)$&$\emptyset$& &\\
\hline

(4-24)&$\langle1,6,12\rangle$&$\langle3,4,6\rangle$&$(3,1)$&$\{\pm(0,1,0)\}$&$\begin{pmatrix}1&0&-4\\0&3&0\\2&0&1\end{pmatrix}$&
$\setlength\arraycolsep{2pt}\begin{array}{l}m\equiv1\Mod3,\\m\not\eq2^{2l+1}(8k+7)\end{array}$\\
\hline

$\setlength\arraycolsep{2pt}\begin{array}{c}\text{(4-26)}\\\alpha_2\not\eq13\\\beta_3\not\eq13\end{array}$&$\begin{pmatrix}4&1&0\\1&7&0\\0&0&6\end{pmatrix}$&$\langle1,6,27\rangle$&$(13,r_4)$&$\emptyset$&~&
$\setlength\arraycolsep{2pt}\begin{array}{l}m\equiv1\Mod3,\\m\equiv r_4\Mod{13},\\m\not\eq2^{2l+1}(8k+7)\end{array}$\\
\hline

(4-32)&$\langle1,2,9\rangle$&$\begin{pmatrix}2&1&0\\1&3&1\\0&1&4\end{pmatrix}$&$(3,0)$&$\emptyset$&~&$\setlength\arraycolsep{2pt}\begin{array}{l}m\equiv0\Mod3,\\m\not\eq2^{2l+1}(8k+7)\end{array}$\\
\hline

(4-33)&$\langle1,2,12\rangle$&$\langle2,3,4\rangle$&$(3,0)$&$\emptyset$&~&
$\setlength\arraycolsep{2pt}\begin{array}{l}m\equiv0\Mod3,\\m\not\eq2^{2l+1}(8k+5)\end{array}$\\
\hline

\multicolumn{7}{r}{
$\begin{array}{ll}
r_1\in\{2,5,6,7,8,11\},&r_2\in\{0,3,5,6\},\\
r_3\in\{0,1,4,5\},&r_4\in\{0,2,5,6,7,8,11\}
\end{array}$}
\end{tabular}
\end{footnotesize}
\end{table}

\vskip0.5pc
\noindent\textbf{Case (4-31)} $(\beta_1,\beta_2,\beta_3)=(1,1,3)$.
It is enough to show that for any $\alpha_1$ such that $4\leq\alpha_1\leq6$, the equation 
$$
3\alpha_1 x^2+y^2+z^2+3t^2=3N+5
$$ 
has an integer solution $(x,y,z,t)\in\mathbb{Z}^4$ such that $yzt\not\equiv0\Mod3$ for any nonnegative integer $N$.

Since the proof of the case $\alpha_1=4$ is quite similar to that of Case (4-1), we only provide proofs of the cases when $\alpha_1\in\{5,6\}$. 

Assume $\alpha_1=5$.
Note that $h(\langle1,1,3\rangle)=1$.  
Every positive integer congruent to $2$ modulo $3$  is represented by $\langle1,1,3\rangle$.
Therefore, the equation 
$$
y^2+z^2+3t^2=3N+5
$$ 
has an integer solution $(y,z,t)=(b,c,d)\in\mathbb{Z}^3$. If $d\not\equiv0\Mod3$, then we are done. 
Assume that $d\equiv0\Mod3$. First, we consider the case when $d\not\eq0$. Then there are a positive integer $s$ and a positive integer $d_1$ not divisible by $3$  such that  $d=3^s\cdot d_1$. 
Now,  by Theorem 9 of \cite{Jones}, there are integers $d_2$ and $a$ such that $3^{2s+1}-15a^2=3d_2^2$  and $d_2\not\equiv0\Mod3$.
Therefore, we have  
$$
b^2+c^2+3(d_2\cdot d_1)^2+15(a\cdot d_1)^2=3N+5.
$$
Now, assume $d=0$.
If $c\equiv0\Mod2$, then $b^2+c^2=b^2+4(\frac c2)^2=b^2+(\frac c2)^2+3(\frac c2)^2$. Therefore, we may assume that $b\equiv c\equiv1\Mod2$. 
Note that $b^2+c^2=2\left(\frac{b+c}{2}\right)^2+2\left(\frac{b-c}{2}\right)^2$ and $b\equiv c \Mod 3$.
If $b\not\eq c$, then there are integers $b_1,c_1$ such that $b^2+c^2=2b_1^2+18c_1^2$ and $c_1\not\eq0$. 
By theorem 9 of \cite{Jones}, there are integers $d_3,a_2$ such that $18c_1^2=3d_3^2+15a_2^2$ and $a_2d_3\not\equiv0\Mod3$. 
Then $b_1^2+b_1^2+3d_3^2+15a_2^2=3N+5$. Finally, assume that $3N+5=2b^2$.
Note that the genus of $f(y,z,t)=(3y+t)^2+(3z+t)^2+3t^2$ consists of
$$
M_f=\begin{pmatrix}5&-2&2\\-2&8&1\\2&1&8\end{pmatrix}\quad\text{and}\quad M_2=\begin{pmatrix}2&1\\1&5\end{pmatrix}\perp\langle27\rangle.
$$
Note that $M_2$ is in the spinor genus of $M_f$.  
Clearly, it suffices to show that $3N+5=2b^2$ is represented bt $M_f$. Assume that $b$ is even.
Since $8$ is represented by $M_f$, $2b^2$ is represented by $M_f$. 
Since $18$ is also represented by $M_f$, we may assume that $b$ is relatively prime to $6$ by a similar reasoning to the above. Since $b>1$, there is a prime $p>3$ dividing $b$. Since $2$ is represented by $M_2$, $2p^2$ and therefore $2b^2$, is represented by $M_f$ by Lemma 2.4 of \cite{poly}.
This completes the proof.

Assume $\alpha_1=6$.
Note that the genus of $f_1(x,y,z)=18x^2+y^2+z^2$ consists of 
$$
M_{f_1}=\langle1,1,18\rangle\quad\text{and}\quad M_3=\langle2\rangle\perp\begin{pmatrix}2&1\\1&5\end{pmatrix}.
$$
Note that $M_3\prec_{2,0}M_{f_1}$. 
Then one may easily show that for a positive integer which is congruent to $2$ modulo $6$ and not of the form $2^{2l+1}(8k+7)$ for any nonnegative integers $l,k$ is represented by $M_f$. 
The rest of the proof is quite similar to that of Case (4-18).

\vskip 0.5pc
\noindent\textbf{Case (4-35)} $(\beta_1,\beta_2,\beta_3)=(1,2,3)$.
It is enough to show that for any $\alpha_1$ such that $5\leq\alpha_1\leq9$, the equation 
$$
3\alpha_1 x^2+y^2+2z^2+3t^2=3N+6
$$ 
has an integer solution $(x,y,z,t)\in\mathbb{Z}^4$ such that $yzt\not\equiv0\Mod3$ for any nonnegative integer $N$.

First, assume $\alpha_1=5$.
By Lemma \ref{reduction}, we may assume that $3N+6\not\equiv0\Mod4$.
If $ N\leq2643$, the one may directly check that such an integer solution exists.
Therefore, we may assume that $n\geq2644$.
Note that the genus of $f(y,z,t)=\frac13((t+3y)^2+2(t+3z)^2+3t^2)$ consists of 
$$
M_f=\begin{pmatrix}2&1&0\\1&3&1\\0&1&4\end{pmatrix}\quad\text{and}\quad M_2=\langle1,2,9\rangle.
$$
One may easily check that any positive integer $m$ which is not of the form $2^{2l+1}(8k+7)$ for any nonnegative integers $l,k$ is represented by $M_f$ or $M_2$.
 Furthermore,  we may check by a direct computation that 
$$
M_2\prec_{2,0}M_f.
$$
Then there is an integer $a\in\{0,1,2,3,5,7\}$ such that both $N+2-5a^2$ and $N+2-5(a+16)^2$ are positive and represented by $M_f$. 
For all possible cases when both $N+2-5a^2$ and $N+2-5(a+16)^2$ are of the form $10s^2$ for some $s\in\mathbb{Z}$, one may directly check that the equation 
$$
15x^2+y^2+2z^2+3t^2=3N+6
$$ 
has an integer solution $(x,y,z,t)\in\mathbb{Z}^4$ such that $yzt\not\equiv0\Mod3$. 
We may assume that at least one of $N+2-5a^2$ and $N+2-5(a+16)^2$  is not of the form $10s^2$ for any $s\in\mathbb{Z}$. 
Therefore, the equation 
$$
15x^2+y^2+2z^2+3t^2=3N+6
$$ 
has an integer solution $(x,y,z,t)=(a,b,c,d)\in\mathbb{Z}^4$ such that $b\equiv c\equiv d\Mod3$ and $3N+6-15a^2$ is not of the form $30s^2$. If $b\equiv c\equiv d\not\equiv0\Mod3$, then we are done. 
Assume $b\equiv c\equiv d\equiv0\Mod3$. 
If $\tau=\frac13\begin{pmatrix}2&2&3\\-1&-1&3\\-1&2&0\end{pmatrix}$, then one may easily check that $\tau(b,c,d)^t=(b_1,c_1,d_1)^t$ is also an integer solution of 
\begin{equation}\label{eq(4-35)}
y^2+2z^2+3t^2=3N+6-15a^2
 \end{equation}
Such that $b_1\equiv c_1\equiv d_1\Mod3$.
Therefore, there is a positive integer $m$ such that $\tau^m(a,b,c)^t=(a_m,b_m,c_m)^t$ is an integer solution of Equation \eqref{eq(4-35)} such that each of whose component is not divisible by $3$ or 
for any positive integer $m$, $\tau^m(a,b,c)^t=(a_m,b_m,c_m)^t$  is an integer solution of Equation \eqref{eq(4-35)}  each of whose component is divisible by $3$.
Since there are only finitely many integer solution of Equation \eqref{eq(4-35)} 
and $\tau$ has an infinite order, the latter is impossible unless $(a,b,c)$ is an eigenvector of $\tau$.
Note that $\pm(0,-3,2)$ are the only integer primitive eigenvectors of $\tau$.
Since $3N+6-15a^2$ is not of the form $30s^2$, the equation 
$$
y^2+2z^2+3t^2=3N+6-15a^2
$$ 
has an integer solution $(y,z,t)\in\mathbb{Z}^3$ such that $yzt\not\equiv0\Mod 3$. 

Assume that $\alpha_1$ is odd. Then every proof is quite similar to that of the case $\alpha_1=5$.

Assume $\alpha_1=6$. 
Note that the genus of $f_1(x,y,z)=18x^2+y^2+2z^2$ consists of 
$$
M_{f}=\langle1,2,18\rangle\quad\text{and}\quad M_2=\begin{pmatrix}3&-1&1\\-1&3&-1\\1&-1&5\end{pmatrix}.
$$
Note that 
$$
M_3\prec_{3,0}M_{f_1}.
$$
 Therefore, a positive integer which is divisible by $3$ and not of the form $4^{l}(8k+7)$ for any nonnegative integers $l,k$ is represented by $M_{f}$.
The rest of the proof is quite similar to that of Case (4-18).
  
Assume $\alpha_1=8$.
By Lemma \ref{reduction}, we may assume that $3N+6\not\equiv0\Mod4$.
If $1\leq N\leq286$, the one may directly check that the equation 
$$
24x^2+y^2+2z^2+3t^2=3N+6
$$ 
has an integer solution such that $yzt\not\equiv0\Mod3$. 
Therefore, we assume that $N\geq277$.
There is an integer $a\in\{3,6\}$ such that $3N+6-24a^2>0$ and furthermore, $3N+6-24a^2$ is represented by $\langle1,2,3\rangle$.
Therefore, the equation 
$$
y^2+2z^2+3t^2=3N+6-24a^2
$$ 
has an integer solution $(y,z,t)=(b,c,d)\in\mathbb{Z}^3$.
We may assume that $b^2+2c^2\not\eq0$ and further assume that $bc\not\equiv0\Mod3$ by theorem 9 of \cite{Jones}.
 If $d\not\equiv0\Mod3$, then we are done. Assume $d\equiv0\Mod3$. 
Since $d^2+8a^2\not\eq0$ and $d^2+8a^2\equiv0\Mod3$, there are integers $e,f$ such that $d^2+8a^2=e^2+8f^2$ and $ef\not\equiv0\Mod3$ by Theorem 9 of \cite{Jones}. This completes the proof of Case (4-35).  \end{proof}

\section{The $61$-theorem of generalized $4$- and $8$-gonal numbers}
In this section, we determine all $n$-ary proper universal mixed sums of generalized $4$- and $8$-gonal numbers for any $n\ge5$.
Furthermore, we give a simple criterion on the universality of an arbitrary mixed sum  of generalized $4$- and $8$-gonal numbers, which is a generalization of the ``$15$-theorem".

\begin{thm}\label{quinary}
There are exactly $707$ quinary proper universal mixed sums of generalized $4$- and $8$-gonal numbers {\rm(}for the list of them, see Table 5.1{\rm)}.
\end{thm}

\begin{proof}
Let $u,v$ be nonnegative integers such that $u+v=5$. 
For positive integers $\alpha_1,\dots,\alpha_u$ and $\beta_1,\dots,\beta_v$, assume that a quinary mixed sum
$$
\Phi=[\![\alpha_1,\dots,\alpha_u,\boldsymbol{\beta_1},\dots,\boldsymbol{\beta_v}]\!]
$$ 
 is proper universal. 
Then by Theorems \ref{ternary} and \ref{quaternary}, there are nonnegative integers $u_1,v_1$ such that $u_1+v_1=4$, and the proper sum 
$$
\Phi'=[\![\alpha_1,\dots,\alpha_{u_1},\boldsymbol{\beta_1},\dots,\boldsymbol{\beta_{v_1}}]\!]
$$ 
of $\Phi$ is one of the sums in Table 5.1 which is not universal(see also Table 4.1).  
For each case, from the universality of $\Phi$ we know that at least one of $\alpha_{u_1+1}$ or $\beta_{v_1+1}$ is less than equal to the truant of the mixed sum  $\Phi'$ given in Table 5.1.
Therefore, we have exactly $708$ candidates of proper quinary universal mixed sums.

 Now, we show that all candidates are, in fact, universal except $[\![1,2,\boldsymbol{5},\boldsymbol{5},\boldsymbol{5}]\!]$. 
The truant of the remaining quinary sum is $20$.
As in Theorem \ref{quaternary}, we may assume that $uv\not\eq0$.
Since the proofs  are quite similar to each other, we only provide, as representative cases,  the proofs of Cases
$$
\text{(5-2)}, \ \text{(5-14)}, \ \text{and} \ \text{(5-16)}.
$$

\begin{table}[h]
\centering
\begin{footnotesize}
\renewcommand{\arraystretch}{0.9}\renewcommand{\tabcolsep}{0.5mm}
\begin{tabular}{|c|c|c|ll|ll|l|}\multicolumn{8}{c}{\textbf{Table 5.1.} Quinary and senary universal mixed sums}\\
\hline
Case&$\Phi'$&$t(\Phi')$&\multicolumn{2}{c|}{Considered already} &\multicolumn{3}{c|}{Universal case} \\
\hline
(5-1)&$[\![1,2,5,5]\!]$&$15$&$6\leq\alpha_5\leq10$,&$6\leq\beta_1\leq10$&\multicolumn{2}{l|}{$\alpha_5,\beta_1=5,~11\leq\alpha_5,\beta_1\leq15$}&~\\
(5-2)&$[\![1,2,5,\boldsymbol{5}]\!]$&$15$&$5\leq\alpha_4\leq10$,&$6\leq\beta_2\leq10$&$11\leq\alpha_4\leq15$,&$\beta_2=5,11\leq\beta_2\leq15$&~\\
(5-3)&$[\![1,2,\boldsymbol{1},\boldsymbol{14}]\!]$&$61$&$2\leq\alpha_3\leq15$,&$\beta_3=15$&$16\leq\alpha_3\leq61$,&$\beta_3=14,16\leq\beta_3\leq61$&~\\
(5-4)&$[\![1,2,\boldsymbol{2},\boldsymbol{7}]\!]$&$30$&$2\leq\alpha_3\leq7$&~&$8\leq\alpha_3\leq30$,&$7\leq\beta_3\leq30$&~\\
(5-5)&$[\![1,2,\boldsymbol{3},\boldsymbol{3}]\!]$&$13$&$2\leq\alpha_3\leq10$,&$4\leq\beta_3\leq10$&$11\leq\alpha_3\leq13$,&$\beta_3=3,11\leq\beta_3\leq13$&~\\
(5-6)&$[\![1,2,\boldsymbol{4},\boldsymbol{14}]\!]$&$60$&$2\leq\alpha_3\leq14$&~&$15\leq\alpha_3\leq60$,&$14\leq\beta_3\leq60$&~\\
(5-7)&$[\![1,2,10,\boldsymbol{5}]\!]$&$20$&$~$&$6\leq\beta_2\leq9$&$10\leq\alpha_4\leq20$,&$\beta_2=5, 10\leq\beta_2\leq20$&~\\
(5-8)&$[\![1,2,\boldsymbol{5},\boldsymbol{5}]\!]$&$15$&$2\leq\alpha_3\leq10$,&$6\leq\beta_3\leq9$&$11\leq\alpha_3\leq15$,&$\beta_3=5,10\leq\beta_3\leq15$&$\beta_3\!\not\eq\!5$\\
(5-9)&$[\![1,2,\boldsymbol{5},\boldsymbol{10}]\!]$&$20$&$2\leq\alpha_3\leq10$&~&$11\leq\alpha_3\leq20$,&$10\leq\beta_3\leq20$&~\\
(5-10)&$[\![2,4,\boldsymbol{1},\boldsymbol{14}]\!]$&$61$&$4\leq\alpha_3\leq15$,&$\beta_3=15$&$16\leq\alpha_3\leq61$,&$\beta_3=14, 16\leq\beta_3\leq61$&~\\
(5-11)&$[\![2,\boldsymbol{1},\boldsymbol{3},\boldsymbol{14}]\!]$&$60$&$2\leq\alpha_2\leq14$&~&$15\leq\alpha_2\leq60$,&$14\leq\beta_4\leq60$&~\\
(5-12)&$[\![3,\boldsymbol{1},\boldsymbol{1},\boldsymbol{7}]\!]$&$30$&$3\leq\alpha_2\leq7$&~&$8\leq\alpha_2\leq30$,&$7\leq\beta_4\leq30$&~\\
(5-13)&\!$[\![\boldsymbol{1},\boldsymbol{1},\boldsymbol{2},\boldsymbol{14}]\!]$\!&$60$&$1\leq\alpha_1\leq14$&~&$15\leq\alpha_1\leq60$,&$14\leq\beta_5\leq60$&~\\
(5-14)&$[\![7,\boldsymbol{1},\boldsymbol{1},\boldsymbol{3}]\!]$&$14$&~&$\beta_4=3,5,6$&$7\leq\alpha_2\leq14$,&$\beta_4=4, 7\leq\beta_4\leq14$&~\\
(5-15)&$[\![\boldsymbol{1},\boldsymbol{1},\boldsymbol{3},\boldsymbol{4}]\!]$&$18$&$1\leq\alpha_1\leq7$,&$\beta_5=5,6$&$8\leq\alpha_1\leq18$,&$\beta_5=4, 7\leq\beta_5\leq18$&~\\
(5-16)&$[\![\boldsymbol{1},\boldsymbol{1},\boldsymbol{3},\boldsymbol{7}]\!]$&$14$&$1\leq\alpha_1\leq7$&~&$8\leq\alpha_1\leq14$,&$7\leq\beta_5\leq14$&~\\
(5-17)&$[\![\boldsymbol{1},\boldsymbol{2},\boldsymbol{3},\boldsymbol{3}]\!]$&$12$&$1\leq\alpha_1\leq9$,&$4\leq\beta_5\leq9$&$10\leq\alpha_1\leq12$,&$\beta_5=3, 10\leq\beta_5\leq12$&~\\
\hline 
\!\!(6-1)&$[\![1,2,\boldsymbol{5},\boldsymbol{5},\boldsymbol{5}]\!]$&20&$1\leq\alpha_3\leq15$,&$6\leq\beta_4\leq15$&$16\leq\alpha_3\leq20$&$\beta_4=5, 16\leq\beta_4\leq20$&~\\
\hline
\end{tabular}
\end{footnotesize}
\end{table}

\noindent\textbf{Case (5-2)} $(\alpha_1,\alpha_2,\alpha_3,\beta_1)=(1,2,5,5)$.
It is enough to show that for any $\alpha_4(\beta_2)$ such that $11\leq\alpha_4\leq15$ $(\beta_2=5$ or $11\leq\beta_2\leq15)$, the equation
$$
\begin{array}{c}
x^2+2y^2+5z^2+\alpha_4t^2+5(3s^2-2s)=N~\\
(x^2+2y^2+5z^2+5(3t^2-2t)+\beta_2(3s^2-2s)=N, ~\text{respectively})
\end{array}
$$
has an integer solution $(x,y,z,t,s)\in\mathbb{Z}^5$ for any nonnegative integer $N$. 

Assume $\alpha_4=11$.
If $N\leq10$, then one may directly check that the equation
$$
x^2+2y^2+5z^2+11t^2+5(3s^2-2s)=N
$$ 
has an integer solution $(x,y,z,t,s)\in\mathbb{Z}^5$.
Since $h(\langle1,2,5\rangle)=1$, one may easily show that every positive integer not of the form $5^{2l+1}(5k+2)$ and $5^{2l+1}(5k+3)$ for any nonnegative integer $l,k$ is represented by $\langle1,2,5\rangle$.
Therefore, we may assume $N\equiv0\Mod5$. Then $N-11$ is represented by $\langle1,2,5\rangle$. If $\alpha_4$ or $\beta_2$ is not divisible by $5$, then the proofs  are quite similar to this. 

Assume $\alpha_4=15$. 
If $N\leq20$, then one may directly check that the equation
$$
x^2+2y^2+5z^2+15t^2+5(3s^2-2s)=N
$$ 
has an integer solution $(x,y,z,t,s)\in\mathbb{Z}^5$. 
Assume $N\geq21$. 
We may assume that $N\equiv0\Mod5$.
Then there are integers $d,e\in\{0,1\}$ such that $N-15d^2-5(3e^2-2e)$ is represented by $\langle1,2,5\rangle$.

If $\beta_2\in\{5,15\}$, then the proof is quite similar to the above.
This completes the proof of Case (5-2).

\vskip 0.5pc
\noindent\textbf{Case (5-14)} $(\alpha_1,\beta_1,\beta_2,\beta_3)=(7,1,1,3)$.
It is enough to show that for any $\alpha_2(\beta_4)$ such that  $7\leq\alpha_2\leq14$ $(\beta_4=4$ or $7\leq\beta_4\leq14)$, the equation
$$
21x^2+3\alpha_2 y^2+z^2+t^2+3s^2=3N+5~(21x^2+y^2+z^2+3t^2+\beta_4 s^2=3N+5+\beta_4)
$$
has an integer solution $(x,y,z,t,s)\in\mathbb{Z}^5$ such that $zts\not\equiv0\Mod3$ $(yzts\not\equiv0\Mod3$, respectively$)$ for any nonnegative integer $N$.

We show that 
\begin{equation}\label{5-14}
x_1^2+y_1^2+3z_1^2+21t_1^2=3N+5 ~\text{with}~ x_1y_1z_1\not\equiv0\Mod3
\end{equation}
has an integer solution $(x_1,y_1,z_1,t_1)\in\mathbb{Z}^4$ for any nonnegative integer $N$ except $14$.
If $N\leq 782$ and $N\not\eq14$, then one may directly check that Equation \eqref{5-14} 
has an integer solution. 
Note that Equation \eqref{5-14} does not have an integer solution for $N=14$, whereas it has an integer solution for $N=4\cdot14=56$. 
Therefore, we may assume that $N\geq783$ and $3N+5\not\equiv0\Mod4$ by similar reasoning of Lemma \ref{reduction}. 
Since $h(\langle1,1,21\rangle)=1$, one may easily show that a positive integer which is congruent to $2$ modulo $3$, not divisible by $7$, and not of the form $4^{l}(8k+3)$ for any nonnegative integers $l,k$ is represented by $\langle1,1,21\rangle$.
Then there is an integer $c\in\{2,4,14,28\}$ such that $3N+5-3c^2>0$, and furthermore, $3N+5-3c^2$ is represented by $\langle1,1,21\rangle$. 
Therefore, the equation 
$$
x_1^2+y_1^2+3z_1^2+21t_1^2=3N+5
$$ 
has an integer solution $(x_1,y_1,z_1,t_1)\in\mathbb{Z}^4$ such that $x_1y_1z_1\not\equiv0\Mod3$ for any nonnegative integer $N$ except $14$. 
The proof of Case (5-14) follows immediately.

 In most of the cases, we can exactly compute the set of all positive integers not represented by $\Phi'$,  denoted by $E(\Phi')$. 
If a proper sum $\Phi'$ of $\Phi$ is a one of the sums in Table 5.2, then the proof is quite similar to that of Case (5-14).

\begin{table}[h]
\centering
\begin{footnotesize}
\renewcommand{\arraystretch}{1}\renewcommand{\tabcolsep}{3mm}
\begin{tabular}{|c|c|c||c|c|c|}\multicolumn{6}{c}{\textbf{Table 5.2.} All cases whose proofs are simair to that of (5-14)}\\
\hline
Case&$\Phi'$&$E(\phi')$&Case&$\Phi'$&$E(\phi')$\\
\hline
(5-1)&$[\![1,2,5,5]\!]$&$\{15\}$&(5-9)&$[\![1,2,\boldsymbol{5},\boldsymbol{10}]\!]$&$E^*$\\
(5-3)&$[\![1,2,\boldsymbol{1},\boldsymbol{14}]\!]$&$\{61\}$&(5-10)&$[\![2,4,\boldsymbol{1},\boldsymbol{14}]\!]$&$\{61\}$\\
(5-4)&$[\![1,2,\boldsymbol{2},\boldsymbol{7}]\!]$&$\{30\}$&(5-11)&$[\![2,\boldsymbol{1},\boldsymbol{3},\boldsymbol{14}]\!]$&$\{60\}$\\
(5-5)&$[\![1,2,\boldsymbol{3},\boldsymbol{3}]\!]$&$\{13\}$&(5-12)&$[\![3,\boldsymbol{1},\boldsymbol{1},\boldsymbol{7}]\!]$&$\{30\}$\\
(5-6)&$[\![1,2,\boldsymbol{4},\boldsymbol{14}]\!]$&$\{60\}$&(5-13)&$[\![\boldsymbol{1},\boldsymbol{1},\boldsymbol{2},\boldsymbol{14}]\!]$&$\{60\}$\\
(5-7)&$[\![1,2,10,\boldsymbol{5}]\!]$&$\{20\}$&(5-15)&$[\![\boldsymbol{1},\boldsymbol{1},\boldsymbol{3},\boldsymbol{4}]\!]$&$\{18\}$\\
(5-8)&$[\![1,2,\boldsymbol{5},\boldsymbol{5}]\!]$&$\{15,20\}$&(5-17)&$[\![\boldsymbol{1},\boldsymbol{2},\boldsymbol{3},\boldsymbol{3}]\!]$&$\{12\}$\\
\hline
\multicolumn{6}{r}{$E^*=\{r\cdot25^s-5 \mid 1\leq r\leq 4,~s\geq1\}$}
\end{tabular}
\end{footnotesize}
\end{table}

\begin{rmk}
(i) In \cite{dick}, Dickson proved that the sum $[\![1,2,5,5]\!]$ represents all positive integers, except $15$.

\noindent (ii) In \cite{octagonal}, it was proved that the sums $[\![\boldsymbol{1},\boldsymbol{1},\boldsymbol{2},\boldsymbol{14}]\!],~ [\![\boldsymbol{1},\boldsymbol{1},\boldsymbol{3},\boldsymbol{4}]\!], ~\text{and}~ [\![\boldsymbol{1},\boldsymbol{2},\boldsymbol{3},\boldsymbol{3}]\!]$ represent all positive integers, except $60$, $18$, and $12$, respectively.   
\end{rmk}

\noindent\textbf{Case (5-16)} $(\beta_1,\beta_2,\beta_3,\beta_4)=(1,1,3,7)$.
It is enough to show that for any $\alpha_1$ such that  $8\leq\alpha_1\leq14$, the equation
$$
3\alpha_1 x^2+y^2+z^2+3t^2+7s^2=3N+12
$$ 
has an integer solution $(x,y,z,t,s)\in\mathbb{Z}^5$ such that $yzts\not\equiv0\Mod3$ for any nonnegative integer $N$. 

Assume $\alpha_1=8$.
If $N\leq1742$, then one may directly check that the equation
$$
24x^2+y^2+z^2+3t^2+7s^2=3N+12
$$ 
has an integer solution $(x,y,z,t,s)\in\mathbb{Z}^5$ such that $yzts\not\equiv0\Mod3$.
Therefore, assume $N\geq 1743$.
Note that the genus of $f(y,z,t)=(3y+t)^2+(3z+t)^2+3t^2$ consists of
$$
M_f=\begin{pmatrix}5&-2&2\\-2&8&1\\2&1&8\end{pmatrix}\quad\text{and}\quad M_2=\begin{pmatrix}2&1\\1&5\end{pmatrix}\perp\langle27\rangle.
$$ 
One may easily show that a positive integer congruent to $2$ modulo $3$ is represented by $M_f$ or $M_2$ by 102:5 of \cite{om} for it is represented  by $M_f$ over $\mathbb{Z}_p$ for any prime $p$.
Note that 
$$
\begin{small}
\begin{array}{ll}
B_f(M_2,13,1)=\{\pm(0,0,1)\},&B_f(M_2,13,3)=\{\pm(0,0,4)\},\\
B_f(M_2,13,4)=\{\pm(0,0,2)\},&B_f(M_2,13,9)=\{\pm(0,0,3)\},\\
B_f(M_2,13,10)=\{\pm(0,0,6)\},&B_f(M_2,13,12)=\{\pm(0,0,5)\}.
\end{array}
\end{small}
$$
In each case, if we define $T=\begin{pmatrix}9&20&0\\-8&1&0\\0&0&13\end{pmatrix}$, then one may easily show that it satisfies all conditions in Theorem \ref{pme}. 
Note that $\pm(0,0,1)$ are the only integral primitive eigenvectors of $T$. Then we have
$$
S_{13,r}\cap Q(M_2)\setminus\{27s^2\mid s\in\mathbb{Z}\}\subset Q(M_f),
$$ 
for any remainder $r$ modulo $13$,  where $r$ is a quadratic residue modulo $13$. 
Therefore, a positive integer which is congruent to $2$ modulo $3$ and quadratic residue modulo $13$ is represented by $M_f$.
There are integers $a\in\{1,2,3,13\}$ and $e\in\{1,2,13\}$ such that $3N+12-7e^2-24a^2>0$ and furthermore, $3N+12-7e^2-24a^2$ is represented by $M_f$.
Therefore, the equation  
$$
y^2+z^2+3t^2=3N+12-7e^2-24a^2
$$ 
has an integer solution $(y,z,t)=(b,c,d)\in\mathbb{Z}^3$ such that $b\equiv c\equiv d\Mod3$. 
This competes the proof.

Assume that $\alpha_1\not\eq13$. Then every proof is quite similar to that of the case $\alpha_1=8$.

Assume $\alpha_1=13$. 
If $N\leq18$, then one may directly check that the equation 
$$
39x^2+y^2+z^2+3t^2+7s^2=3N+12
$$ 
has an integer solution $(x,y,z,t,s)\in\mathbb{Z}^5$ such that $yzts\not\equiv0\Mod3$. 
Therefore, assume $N\geq19$. 
Note that $h(\langle1,1,3\rangle)=1$.
There are integers $e\in\{1,2\}$ and $a\in\{0,1\}$ such that $3N+12-7e^2-39a^2>0$ and $3N+12-7e^2-39a^2$ is congruent to $5$ or $8$ modulo $12$.
Therefore, the equation 
$$
y^2+z^2+3t^2=3N+12-7e^2-39a^2
$$ 
has an integer solution $(y,z,t)=(b,c,d)\in\mathbb{Z}^3$ such that $bc\not\equiv0\Mod3$ and
$c\equiv d\Mod2$.
If $d\not\equiv0\Mod3$, then we are done.
Assume $d\equiv0\Mod3$. 
Then 
$c^2+3d^2=\left(\frac{c+3d}{2}\right)^2+3\cdot\left(\frac{c-d}{2}\right)^2$ and $\left(\frac{c+3d}{2}\right)\cdot\left(\frac{c-d}{2}\right)\not\equiv0\Mod3$.
This completes the proof of Case (5-16).
\end{proof}

\begin{thm}\label{senary}
There are exactly $11$ senary proper universal mixed sums of generalized $4$- and $8$-gonal numbers {\rm(}for the list of them, see Table 5.1{\rm)}.
\end{thm}

\begin{proof}
Let $u,v$ be nonnegative integers such that $u+v=6$. 
For positive integers $\alpha_1,\dots,\alpha_u$ and $\beta_1,\dots,\beta_v$, assume that a senary mixed sum
$$
\Phi=[\![\alpha_1,\dots,\alpha_u,\boldsymbol{\beta_1},\dots,\boldsymbol{\beta_v}]\!]
$$ 
of generalized $4$- and $8$-gonal numbers is proper universal.
Then by Theorems \ref{ternary}, \ref{quaternary}, and \ref{quinary}, $\Phi$ has a proper sum $[\![1,2,\boldsymbol5,\boldsymbol5,\boldsymbol5]\!]$ (see Table 5.1).
From the universality of $\Phi$, we know that at least one of $\alpha_3$ or $\beta_4$ is less than equal to $20$ that is the truant of $[\![1,2,\boldsymbol5,\boldsymbol5,\boldsymbol5]\!]$.

\vskip 0.5pc
\noindent\textbf{Case (6-1)} $(\alpha_1,\alpha_2,\beta_1,\beta_2,\beta_3)=(1,2,5,5,5)$.
It is enough to show that for any $\alpha_3(\beta_4)$ such that $16\leq\alpha_3\leq20$ $(\beta_4=5$ or $16\leq\beta_4\leq20)$, the equation
$$
\begin{small}
\begin{array}{c}
x^2+2y^2+\alpha_3 z^2+5(3t^2-2t)+5(3s^2-2s)+5(3u^2-2u)=N\\
(x^2+2y^2+5(3z^2-2z)+5(3t^2-2t)+5(3s^2-2s)+\beta_4(3u^2-2u)=N,~\text{respectively})
\end{array}
\end{small}
$$
has an integer solution $(x,y,z,t,s,u)\in\mathbb{Z}^6$ for any nonnegative integer $N$.
Since the mixed sum $[\![1,2,\boldsymbol5,\boldsymbol5]$ of generalized $4$- and $8$-gonal numbers represents all nonnegative integers except $15$ and $20$ (see Table 5.2),  the quinary mixed sum $[\![1,2,\boldsymbol5,\boldsymbol5,\boldsymbol5]\!]$ of generalized $4$- and $8$-gonal numbers represents all positive integers except $20$. 
The proof follows immediately.
\end{proof}

Finally, we give an effective criterion on the universality of an arbitrary mixed sum  of generalized $4$- and $8$-gonal numbers. 
\begin{thm}
Let $\alpha_1,\dots,\alpha_u$ and $\beta_1,\dots,\beta_v$ be positive integers. The mixed sum 
$$
\Phi=\alpha_1P_{4}(x_1)+\cdots+\alpha_uP_{4}(x_u)+\beta_1P_8(y_1)+\cdots+\beta_vP_{8}(y_v) 
$$ 
of generalized $4$- and $8$-gonal numbers is universal if and only if it represents the  integers
$$
1,\ 2,\ 3,\ 4,\ 5,\ 6,\ 7,\ 8,\ 9,\ 10,\ 12,\ 13,\ 14,\ 15,\ 18,\ 20,\ 30,\ 60,\ \text{and} \ \ 61.
$$
\end{thm}

\begin{proof}
Assume that $\Phi=[\![\alpha_1,\dots,\alpha_{u},\boldsymbol{\beta_1},\dots,\boldsymbol{\beta_{v}}]\!]$ represents above $19$ integers. 
Then one may easily show that  there are nonnegative integers $u_1,v_1$ such that $u_1\leq u$, $v_1\leq v$, and furthermore, the proper sum 
$$
\widetilde{\Phi}=[\![\alpha_1,\dots,\alpha_{u_1},\boldsymbol{\beta_1},\dots,\boldsymbol{\beta_{v_1}}]\!]
$$ 
of $\Phi$ is one of the mixed sums of Tables 3.1, 4.1, and  5.1.  
Then the the proper sum $\widetilde{\Phi}$ of $\Phi$ is universal 
by Theorems \ref{ternary}, \ref{quaternary}, \ref{quinary}, and \ref{senary}. Therefore, the mixed sum $\Phi$ of generalized $4$- and $8$-gonal numbers is universal.
\end{proof}

\end{document}